\documentclass[preprint,12pt]{elsarticle}

\usepackage{amssymb,amsmath,lineno,hyperref,geometry}
\usepackage{color}
\usepackage{amssymb}
\numberwithin{equation}{section}
\newtheorem{theorem}{Theorem}[section]
\newtheorem{lemma}{Lemma}[section]
\newtheorem{remark}{Remark}[section]

\newtheorem{definition}{Definition}[section]

\newtheorem{example}{Example}[section]

\newtheorem{proposition}{Proposition}[section]

\journal{arXiv}

\begin{document}

\newcommand\tbbint{{-\mkern -16mu\int}}
\newcommand\tbint{{\mathchar '26\mkern -14mu\int}}
\newcommand\dbbint{{-\mkern -19mu\int}}
\newcommand\dbint{{\mathchar '26\mkern -18mu\int}}
\newcommand\bint{
{\mathchoice{\dbint}{\tbint}{\tbint}{\tbint}}
}
\newcommand\bbint{
{\mathchoice{\dbbint}{\tbbint}{\tbbint}{\tbbint}}
}

\begin{frontmatter}



\title{{\bf Global existence and nonexistence analyses for a magnetic fractional pseudo-parabolic equation}}

\author[1]{Jiazhuo Cheng}
\ead{chengjzh5@mail2.sysu.edu.cn}
\author[1,2]{Qiru Wang\corref{cor1}
}
\ead{mcswqr@mail.sysu.edu.cn}

\cortext[cor1]{Corresponding author}
\address[1]{School of Mathematics, Sun Yat-sen University, Guangzhou 510275, Guangdong, PR China}
\address[2]{School of Mathematics(Zhuhai), Sun Yat-sen University, Zhuhai 519082, Guangdong, PR China}




\begin{abstract}
In this paper, we study the initial-boundary value problem for a pseudo-parabolic equation in magnetic fractional Orlicz-Sobolev spaces. First, by employing the imbedding theorems, the theory of potential wells and the Galerkin method, we prove the existence  and uniqueness of global solutions with subcritical initial energy, critical initial energy and supercritical initial energy, respectively. Furthermore, we prove the decay estimate of global solutions with sub-sharp-critical initial energy, sharp-critical initial energy and supercritical initial energy, respectively. Specifically, we need to analyze the properties of $\omega$-limits of solutions for supercritical initial energy. Next, we establish the finite time blowup of solutions with sub-sharp-critical initial energy and sharp-critical initial energy, respectively.  Finally, we discuss the convergence relationship between the global solutions of the evolution problem and the ground state solutions of the corresponding stationary problem.
\end{abstract}


\end{frontmatter}

\newpage

\tableofcontents

\newpage


\section{Introduction}
\label{section1}

In this paper, we consider the following initial boundary value problem of fractional  non-linear  pseudo-parabolic equation
\begin{equation}\label{1.1}
\left\{\begin{split}
&u_t+\left(-\Delta\right)^s u_t+\left(-\Delta_g^A\right)^su=\left|u\right|^{p-1}u,\;\;\;\;\;\;\;\;\;\;\;\;\;\;\;\;\;0<t<T,x\in\Omega,\\
&u\left(x,t\right)=0,\;\;\;\;\;\;\;\;\;\;\;\;\;\;\;\;\;\;\;\;\;\;\;\;\;\;\;\;\;\;\;\;\;\;\;\;\;\;\;\;\;\;\;\;\;\;\;\;\;0<t<T,\;x\in\mathbb{R}^N\backslash\Omega,\\
&u\left(x,0\right)=u_0(x),\;\;\;\;\;\;\;\;\;\;\;\;\;\;\;\;\;\;\;\;\;\;\;\;\;\;\;\;\;\;\;\;\;\;\;\;\;\;\;\;\;\;\;\;\;\;\;\;\;\;\;\;\;\;\;\;\;\;\;\;\;\;\;\;\;\;\;x\in\Omega,
\end{split}\right.
\end{equation}
where $\Omega\subset\mathbb{R}^N\left(N\geq1\right)$ is a bounded domain with smooth boundary, $T\in\left(0,\infty\right]$. The fractional operators $\left(-\Delta\right)^s$ and $\left(-\Delta_g^A\right)^s$ are defined as
\begin{equation*}
\left(-\Delta\right)^su\left(x\right):=2\mathrm P.\mathrm V.\int_{\mathbb{R}^N}\frac{u\left(x\right)-u\left(y\right)}{\left|x-y\right|^{N+2s}}\mathrm dy,\;\;\;x\in\mathbb{R}^N,
\end{equation*}
and
\begin{equation}\label{1.2}
\left(-\Delta_g^A\right)^su\left(x\right):=2\mathrm P.\mathrm V.\int_{\mathbb{R}^N}\frac{g\left(\left|D_s^{A}u\left(x, y\right)\right|\right)}{\left|D_{s}^{A}u\left(x,y\right)\right|}D_{s}^{A}u\left(x,y\right)\frac{\mathrm{d}y}{\left|x-y\right|^{ N+s}},\;\;\;x\in\mathbb{R}^{N},
\end{equation}
where $A:\mathbb{R}^N\rightarrow\mathbb{R}^N$ is a given vector field, $s\in\left(0,1\right)$, $u:\mathbb{R}^N\rightarrow\mathbb{C}$ is a smooth function, $\mathrm P.\mathrm V.$ stands for ``in principal value", $g=G'$ is the derivative of an $N$-function (see Section 2 for details regarding $N$-functions), and
\begin{equation*}
  D_s^Au\left(x,y\right)=\frac{u\left(x\right)-e^{i\left(x-y\right)\cdot A\left({\frac{x+y}2}\right)}u\left(y\right)}{\left|x-y\right|^s}.
\end{equation*}

The operator $\left(-\Delta_g^A\right)^s$ is called the magnetic fraction $g$-Laplacian operator, which appears as a gradient of the non-local magnetic energy function
\begin{equation*}
  \rho_{s,G}^A\left(u\right):=\int_{\mathbb{R}^{2N}}G\left(\left|D_s^Au\left(x,y\right)\right|\right)\mathrm{d}\mu,
\end{equation*}
i.e.,
\begin{equation*}
\begin{split}
&\lim_{t\rightarrow0}\frac{\rho_{s,G}^A\left(u+tv\right)-\rho_{s,G}^A\left(u\right)}t=\Re\left(\int_{\mathbb{R}^{2N}}\frac{g\left(\left|D_s^Au\left(x,y\right)\right|\right)}{\left|D_s^Au\left(x,y\right)\right|}D_s^Au\left(x,y\right)\overline{D_s^Av\left(x,y\right)}\mathrm{d}\mu\right)\\
&\;\;\;\;\;\;\;\;\;\;\;\;\;\;\;\;\;\;\;\;\;\;\;\;\;\;\;\;\;\;\;\;\;\;\;\;\;\;\;\;\;=\left\langle\left(-\Delta_g^A\right)^su,v\right\rangle,
\end{split}
\end{equation*}
where $\Re\left(z\right)$ denotes its real part for $z\in\mathbb{C}$, $\mathrm d\mu=\mathrm{d}x\mathrm{d}y/\left|x-y\right|^N$. On the other hand, the operator $\left(-\Delta_g^A\right)^s$ can be seen as an approximate of local operator $-\Delta _{\overline g}^A$ appearing as a gradient of the energy functional
\begin{equation*}
I_{\overline G}^A\left(u\right)=\int_{\mathbb{R}^N}\overline G\left(\left|\nabla u-iAu\right|\right)\mathrm{d}x,
\end{equation*}
where $\overline g={\overline G}'$ and $\overline G$ is the spherical limit of $G$ (see \cite{Bonder2}). For example, the spherical limit of $G$ is $\overline G\left(t\right)=K_{N,p}\left|t\right|^p$ when $G\left(t\right)=\left|t\right|^p$ and $p>1$, where $K_{N,p}:=\int_{\mathbb{S}^{N-1}}\left|x_n\right|^p\mathrm dp$.
For more examples, please refer to \cite{Bonder}. The approximation
is in the following sense (see \cite[Theorem 6.17]{Bonder}): let $u_s$ be the unique solution of
\begin{equation*}
  \left(-\Delta_g^A\right)^su=f\;\;\;\mathrm{in}\;\Omega\subset\mathbb{R}^N,\;u=0\;\;\;\mathrm{on}\;\mathbb{R}^N\backslash\Omega
\end{equation*}
for $s\in\left(0,1\right)$ and  appropriate $f$, then $u_s\rightarrow u$ in $L^G\left(\Omega\right)$ for $s\rightarrow 1$, and $u$ is the weak solution to
\begin{equation*}
-\Delta_{\overline g}^Au=f\;\;\;\mathrm{in}\;\Omega\subset\mathbb{R}^N,\;u=0\;\;\;\mathrm{on}\;\partial\mathrm\Omega.
\end{equation*}
When $G\left(t\right)=\left|t\right|^2$, the derivative of the corresponding energy functional $I_{\overline G}^A$ generates the (real part of the) magnetic Schr$\mathrm{\ddot{o}}$dinger operator, defined as
\begin{equation*}
-\left(\nabla-iA\right)^2u\left(x\right)=-\Delta u\left(x\right)+2iA\left(x\right)\cdot\nabla u+\left|A\left(x\right)\right|^2u+iu\mathrm{div}A\left(x\right).
\end{equation*}
The magnetic Schr$\mathrm{\ddot{o}}$dinger operator has been extensively studied in the past few decades, see \cite{Bartsch,Cingolani,Schindler,Arioli,Pablo}.

In the case the magnetic field $A=0$ in (\ref{1.2}), $\left(-\Delta_g^{A}\right)^su\left(x\right)$ is abbreviated as $\left(-\Delta_g\right)^su\left(x\right)$ (see \cite{Bonder}), which is defined as
\begin{equation*}
\left(-\Delta_g\right)^su\left(x\right):=C_{n,s}\mathrm P.\mathrm V.\int_{\mathbb{R}^N}g\left(\frac{\left|u\left(x\right)-u\left(y\right)\right|}{\left|x-y\right|^s}\right)\frac{u\left(x\right)-u\left(y\right)}{\left|u\left(x\right)-u\left(y\right)\right|}\frac{\mathrm dy}{\left|x-y\right|^{N+s}},\;\;\;x\in\mathbb{R}^N,
\end{equation*}
where $u$ is a real function and sufficiently smooth. The  fractional $g$-Laplacian $\left(-\Delta_g\right)^su\left(x\right)$ generalizes a large class of familiar non-local operators. For instance, when $g\left(t\right)=t^{p-1}$, then $\left(-\Delta_g\right)^su\left(x\right)$ becomes the  \emph{fractional $p$-Laplace operator} $\left(-\Delta_p\right)^su\left(x\right)$, see \cite{Nezza,Puhst,Pana,Liao,Cheng1}. Furthermore, if $g\left(t\right)=t$, then $\left(-\Delta_g\right)^su\left(x\right)$  reduces to the well known \emph{fractional Laplace operator}  $\left(-\Delta\right)^su\left(x\right)$,  we refer the reader to the following list \cite{Stein,Landkof,Laskin,Ambrosio,Daoud} and to the references  \cite{Bertoin,Laskin2,Metzler,Gilboa,Caffarelli} for the physical background. Non-local operators can appear in modeling different physical situations such as anomalous diffusion and quasi-geostrophic flow, turbulence and water waves, molecular dynamics and relativistic quantum mechanics of stars (see \cite{Bouchaud,Constantin}), as well as appear in mathematical finance \cite{Applebaum}, elasticity problems \cite{Signorini}, phase transition problems \cite{Alberti}, and crystal dislocation structures \cite{Toland}. Usually, non-local generalizations of Laplacian (linear and nonlinear) are tested using smooth kernels, and fractional Laplacian is represented as integral operator on the entire $\mathbb{R}^N$. In this work, we focus on non-local operators acting on bounded domains. These bounded domains correspond to regional fractional Laplacian,  which can be interpreted  as a non-local version of the Laplacian equipped with Dirichlet boundary conditions, see \cite{Puhst}. Here, the operator (\ref{1.2})  constitutes a fractional relativistic generalization of magnetic Laplacian, see \cite{Ichinose,Ichinose2,Avenia}.

Problem (\ref{1.1}) describes a variety of important physical processes, such as the unidirectional propagation of nonlinear, dispersive, long waves \cite{Benjamin,Ting} and the aggregation of population \cite{Padron}. The classic pseudo-parabolic equation is as follows
\begin{equation}\label{pseudo-parabolic equation}
u_t-\Delta u_t=F\left(x,t,\nabla u,\Delta u\right),
\end{equation}
which has been investigated by Showalter and Ting \cite{Showalter} for $F\left(x,t,\nabla u,\Delta u\right)=\Delta u$,  and they proved the existence and uniqueness of solutions. From 2013 to 2018, the authors in \cite{Xu2,Xu1,Liu3,Xu3} considered the initial boundary value problem  (\ref{pseudo-parabolic equation}) with $F\left(x,t,\nabla u,\Delta u\right)=\Delta u+u^p$, and proved the global existence and asymptotic behavior of solutions with subcritical and critical initial energy $J(u_0)\leq d$, and also gave the global nonexistence of solutions under supercritical initial energy $J(u_0)>d$ through comparison principle; further, the upper bound of the blowup time of supercritical initial energy was estimated. Due to the fact that $u(x,y)$ is not required to be non-negative in mathematics, some authors \cite{Wang,Cheng,Cheng1,Cheng2}  replace $u^q$ with $\left|u\right|^{q-1}u$ in the problem (\ref{pseudo-parabolic equation}) (see \cite{Hu}). In 2023, Cheng and Wang \cite{Cheng1} studied the following fractional  pseudo-parabolic equation
\begin{equation*}
u_t+\left(-\Delta\right)_p^su+\left(-\Delta\right)^su_t=\left|u\right|^{q-1}u-\bbint_\Omega\left|u\right|^{q-1}u\mathrm dx
\end{equation*}
and proved the existence, uniqueness and decay estimate of global solutions and the blowup phenomena of solutions with subcritical initial energy and critical initial energy by constructing a family of potential wells (see \cite{Sattinger,Liu1,Liu2}), and established the existence, uniqueness and asymptotic behavior of global solutions with supercritical initial energy by further analyzing the properties of $\omega$-limits of solutions.

In addition to the above special conditions that the global solutions converging to $0$ as $t\rightarrow\infty$ when $u_0$ satisfies some special conditions, recently, the  researchers in \cite{Zhou,Cao2} further studied the asymptotic behavior of general global solutions and found that it is related to the ground state solution of stationary problems. In 2021, Cao and Zhao  \cite{Cao2}  studied the following pseudo-parabolic equation
\begin{equation*}
  u_t-\Delta u_t-M(\left\|\nabla u\right\|_p^p)\mathrm{div}(\vert\nabla u\vert^{p-2}\nabla u)=\vert u\vert^{q-1}u-\frac1{\vert\Omega\vert}\;\int_\Omega\;\vert u\vert^{q-1}u \mathrm{d}x,
\end{equation*}
proved the convergent relationship between the global solutions of the evolution problem and the ground state solutions of the corresponding stationary problem when $p\geq2,\;2p-1<q<p^\ast-1,\;M(s)=a+bs\;\mathrm{with}\;a>0\;\mathrm{and}\;b>0.$

Contrary to the classical fractional Laplacian Sch$\ddot{\mathrm o}$dinger equation that is widely investigated, the situation seems to be in a developing state when the new fractional $g$-Laplacian is presented. In this context, the natural setting for studying problem (\ref{1.1}) is fractional Orlicz-Sobolev spaces. Currently, as far as we know, some results for fractional Orlicz-Sobolev spaces and fractional $g$-Laplacian operator are obtained in \cite{Bonder,Bahrouni,Napoli,Bonder2,Bonder3,Salort}.  Therefore, we combine the potential wells theory with the Galerkin method to establish the global existence, asymptotic behavior, and finite time blowup of the solutions in fractional Orlicz-Sobolev spaces. Furthermore, we prove some convergence relations between the global solutions of (\ref{1.1}) and the ground state solutions of the following stationary problem
\begin{equation}\label{1.7}
\left\{\begin{split}
&\left(-\Delta_g^A\right)^su=\left|u\right|^{p-1}u,\;\;\;\;\;\;\;\;\;\;\;\;\;\;\;\;\;\;\;\;\;\;\;x\in\Omega,\\
&u=0,\;\;\;\;\;\;\;\;\;\;\;\;\;\;\;\;\;\;\;\;\;\;\;\;\;\;\;\;\;\;\;\;\;\;\;\;\;\;\;\;x\in\mathbb{R}^N\backslash\Omega.\\
\end{split}\right.
\end{equation}

This paper is organized as follows:

(i) In Section \ref{Preliminaries and main results}, we give some notations, necessary definitions and lemmas about the basic properties of the related functionals, sets and spaces.  We also  give the main assumptions and main results of this paper.

(ii) In Section \ref{Local solutions}, we prove the existence and uniqueness of local solutions by the Galerkin method and Contraction Mapping Principle.

(iii) In Section \ref{Global existence and uniqueness}, we prove the existence and uniqueness of global solutions with subcritical initial energy, critical initial energy and supercritical initial energy, respectively.

(iv) In Section \ref{Asymptotic behavior}, we discuss the decay estimate of global solutions with sub-sharp-critical initial energy, sharp-critical initial energy and supercritical initial energy, respectively.

(v) In Section \ref{Blowup}, we establish the blowup of weak solutions with sub-sharp-critical initial energy and sharp-critical initial energy, respectively.

(vi) In Section \ref{some convergence relations}, we establish convergence relations between the global solutions of (\ref{1.1}) and the ground state solutions of (\ref{1.7}).

(vii) In Section \ref{Conclusions}, we make some conclusions.

\section{Preliminaries and main results}\label{Preliminaries and main results}

\subsection{$N$-functions and basic properties}
Let's recall the definition of the $N$-function.
\begin{definition}[$N$-function]\label{$N$-function}
Function $G:\left[0,\infty\right)\rightarrow\mathbb{R}$ is called an $N$-function, if it can be written as
\begin{equation}\label{2.1.1}
  G\left(t\right)=\int_0^tg\left(\tau\right)\mathrm{d}\tau,
\end{equation}
where the function $g$ is positive for $t>0$, right-continuous for $t\geq0$, non-decreasing and $g\left(0\right)=0$ and $g\left(t\right)\rightarrow\infty$ as $t\rightarrow\infty$.
\end{definition}

From \cite{Krasnoselskii}, an $N$-function satisfies the following important properties (equivalent to the definition of the $N$-function):\\
(i) $G$ is even, continuous, convex, increasing and $G\left(0\right)=0$.\\
(ii) $G$ is super-linear at zero and at infinite, i.e., $\underset{x\rightarrow0}{\lim}\frac{G\left(x\right)}x=0$ and $\underset{x\rightarrow\infty}{\lim}\frac{G\left(x\right)}x=\infty$.

Another important property of the $N$-function $G$ satisfies the $\Delta_2$  condition as follows:
\begin{definition}\label{2 condition}
The $N$-function $G$ satisfies the $\Delta_2$  condition, if there exists $C>2$ such that
\begin{equation*}
G\left(2x\right)\leq CG\left(x\right)\;\;\;\mathrm{for}\;\mathrm{all}\;x\in{\mathbb{R}}_+,
\end{equation*}
where ${\mathbb{R}}_+$ represents the set of positive real numbers.
\end{definition}

\begin{example}[\cite{Krasnoselskii}]\label{example 1}
Some examples of functions satisfying Definition \ref{2 condition} are as follows:
\begin{equation*}
\begin{split}
&\left(i\right)\;G\left(t\right)=t^q,\;t\geq0,\;q>1;\\
&\left(ii\right)\;G\left(t\right)=t^q\left(\left|\log\left(t\right)\right|+1\right),\;t\geq0,\;q>1;\\
&\left(iii\right)\;G\left(t\right)=\left(1+\left|t\right|\right)\log\left(\left|t\right|+1\right)-\left|t\right|;\\
&\left(iv\right)\;G\left(t\right)=t^q\chi_{\left(0,1\right]}\left(t\right)+t^p\chi_{\left(1,\infty\right)}\left(t\right),\;t\geq0,\;q,p>1.
\end{split}
\end{equation*}
\end{example}

\begin{remark}\label{remark 1}
Without losing generality, $G$ can be normalized such that $G\left(1\right)=1$.
\end{remark}

As shown in \cite[Theorem 4.1, Chapter 1]{Krasnoselskii},  $G$ satisfies the $\Delta_2$ condition is equivalent to
\begin{equation}\label{2.1.1(1)}
\frac{tg\left(t\right)}{G\left(t\right)}\leq q^+,\;\;\;\forall t>0,
\end{equation}
for some $q^+>1$.

The $\Delta_2$ condition of the \emph{complementary function} of $G$ is defined as follows
\begin{equation}\label{2.1.2}
\widetilde G\left(s\right):=\sup_{t>0}\left\{st-G\left(t\right)\right\}.
\end{equation}
The following expression formula for $\widetilde G$ holds:
\begin{equation*}
\widetilde G\left(s\right)=\int_0^sg^{-1}\left(\tau\right)\mathrm d\tau.
\end{equation*}
Here, $g^{-1}$ represents the right-continuous inverse of the function $g$ that appears in (\ref{2.1.1}). In particular, if $G\left(t\right)=t^p\left(p>1\right)$, $\widetilde G$ plays the same role as the conjugate exponential function.

From (\ref{2.1.2}), it is directly concluded that the following Young-type inequality is established
\begin{equation*}
at\leq G\left(t\right)+\widetilde G\left(a\right)\;\;\;\mathrm{for}\;\mathrm{every}\;a,t\geq0.
\end{equation*}

From \cite[Theorem 4.3, Chapter 1]{Krasnoselskii}, a necessary and sufficient condition for the $N$-function $\widetilde G$ complementary to $G$ to satisfy the $\Delta_2$ condition  is as follows
\begin{equation*}
q^-\leq\frac{tg\left(t\right)}{G\left(t\right)},\;\;\;\forall t>0,
\end{equation*}
for $q^->1$.

Based on the above analysis and our calculation needs, we make the following strong assumptions about the $N$-function
\begin{equation}\label{2.1.4}
1<q^-\leq\frac{tg\left(t\right)}{G\left(t\right)}\leq q^+,\;\;\;\forall t>0.
\end{equation}
Then, $\widetilde G$ and $G$ both satisfy the $\Delta_2$ condition. By using $g\left(t\right)=-g\left(-t\right)$ for $t<0$, we can extend (\ref{2.1.4}) holding for all $t\in\mathbb{R},\;t\neq0$. Another assumption that used in this  manuscript is as follows
\begin{equation}\label{2.1.5}
q^--1\leq\frac{tg'\left(t\right)}{g\left(t\right)}\leq q^+-1,\;\;\;\forall t>0.
\end{equation}
Similar to (\ref{2.1.4}), we have (\ref{2.1.5}) established for all $t\in\mathbb{R},\;t\neq0$. It is easy to know that (\ref{2.1.5})  implies (\ref{2.1.4}). For example, the $N$-function $g\left(t\right)=t^{q-1}$ satisfies (\ref{2.1.5}) for $q^-\leq q\leq q^+$.

According to (\ref{2.1.4}) and (\ref{2.1.5}), we have the following comparison relation for power functions and $N$-functions
\begin{equation}\label{2.1.6}
\min\left\{a^{q^--1},a^{q^+-1}\right\}g\left(b\right)\leq g\left(ab\right)\leq\max\left\{a^{q^--1},a^{q^+-1}\right\}g\left(b\right),
\end{equation}
and
\begin{equation}\label{2.1.7}
\min\left\{a^{q^-},a^{q^+}\right\}G\left(b\right)\leq G\left(ab\right)\leq\max\left\{a^{q^-},a^{q^+}\right\}G\left(b\right),
\end{equation}
where $a>0,\;b>0$. The proof process of (\ref{2.1.6}) and (\ref{2.1.7}) can refer to \cite[Lemma 2.5]{Bonder}.

To end this subsection, we introduce one of the important symbols involved in embedding lemmas in Subsection 2.3.
\begin{definition}\label{$A$and$B$}
Let $A$ and $B$ be two $N$-functions, we say that $B$ is essentially stronger than $A$,  denoted
by $A\ll B$, if there exists $x_a\geq0$ for any $a>0$ such that
\begin{equation*}
A\left(x\right)\leq B\left(ax\right),\;x\geq x_a.
\end{equation*}
\end{definition}
Obviously Definition \ref{$A$and$B$} is equivalent to $\underset{t\rightarrow+\infty}{\lim}\frac{A\left(kt\right)}{B\left(t\right)}=0$ for any $k>0$ (see \cite[Theorem 2]{Rao}).

\subsection{Main assumptions}
Throughout this paper, let $C$ represent generic positive constant, which may change from line to line. We assume that  the real numbers $s$, $p$, $q$ and the $N$-function $G$ satisfy the following conditions:
\begin{equation*}
\begin{split}
&\left(H_1\right)\;0<s<1<q^-\leq q^+<p+1<q^{-\ast},\;\max\left\{\frac{tg\left(t\right)}{G\left(t\right)}-1,1\right\}<p,\;\forall t>0;\\
&\left(H_2\right)\;\int_0^1\frac{G^{-1}\left(t\right)}{t^{\frac{N+s}N}}\mathrm dt<\infty\;\mathrm{and}\;\int_1^{+\infty}\frac{G^{-1}\left(t\right)}{t^{\frac{N+s}N}}\mathrm dt=\infty,\;\mathrm{where}\;0<s<1;\\
&\left(H_3\right)\;\mathrm{The}\;\mathrm{function}\;t\rightarrow G\left(\sqrt t\right)\;\mathrm{is}\;\mathrm{convex};\\
\end{split}
\end{equation*}
where $q^-:=\underset{t>0}{\inf}\frac{tg\left(t\right)}{G\left(t\right)}$, $q^+:=\underset{t>0}{\sup}\frac{tg\left(t\right)}{G\left(t\right)}$, $q^{-\ast}:=\left\{\begin{array}{l}\frac{Nq^-}{N-sq^-},\;\;\;\mathrm{if}\;sq^-<N,\\\infty,\;\;\;\mathrm{if}\;sq^-\geq N.\end{array}\right.$

If $\left(H_2\right)$  is satisfied, the Sobolev conjugate $N$-function $G^\ast$ of $G$ can be defined as
\begin{equation}\label{Sobolev conjugate}
  \left(G^\ast\right)^{-1}\left(t\right)=\int_0^t\frac{G^{-1}\left(w\right)}{w^{\left(N+s\right)/N}}\mathrm dw.
\end{equation}
The above conjugate function will be applied to the embedding result in the  Orlicz-Sobolev spaces.
\begin{example}\label{example 2}
Some examples of functions that satisfy the above assumptions are as follows (see \cite[section 1, Remark 1.1]{Bahrouni} and \cite[subsection 2.2, Remark 1]{Pablo}):
\begin{equation*}
\begin{split}
&\left(\mathrm i\right)\;\mathrm{if}\;G\left(t\right)=t^p,\;p>1,\;\mathrm{then}\;(H_2)\;\mathrm{holds}\;\mathrm{for}\;sp<N\;\mathrm{and}\;(H_3)\;\mathrm{holds}\;\mathrm{for}\;p\geq2;\\
&\left(\mathrm{ii}\right)\;\mathrm{if}\;G\left(t\right)=t^p+t^q,\;p\geq q>1,\;\mathrm{then}\;(H_2)\;\mathrm{holds}\;\mathrm{for}\;sp<N\;\mathrm{and}\;(H_3)\;\mathrm{holds}\;\mathrm{for}\;q\geq2;\\
&\left(\mathrm{iii}\right)\;\mathrm{if}\;G\left(t\right)=\int_0^t\left(p\left|s\right|^{p-2}s\left(\left|log\left(s\right)\right|+1\right)+\frac{\left|s\right|^{p-2}s}{1+\left|s\right|}\right)\mathrm ds,\;\mathrm{then}\;(H_2)\;\mathrm{holds}\;\mathrm{for}\\
&s\left(p+1\right)<N\;\mathrm{and}\;(H_3)\;\mathrm{holds}\;\mathrm{for}\;q>2.\\
\end{split}
\end{equation*}
\end{example}

\subsection{Lebesgue and Orlicz-Sobolev spaces}
In this subsection, we shall introduce some definitions of Lebesgue and Orlicz-Sobolev spaces.

Let $\Omega$ be an open set in $\mathbb{R}^N$ and $Q=\left(\mathbb{R}^N\times\mathbb{R}^N\right)\backslash\left(\Omega^c\times\Omega^c\right),\;\Omega^c=\mathbb{R}^N\backslash\Omega$, then $\Omega\times\Omega$ is strictly contained in $Q$. For $1\leq p\leq\infty$, denote the $L^p\left(\Omega,\mathbb{C}\right)$ norm by ${\left\|\cdot\right\|}_p$.  Define the set
 \begin{equation*}
W_0^{s,2}\left(\Omega,\mathbb{C}\right):=\left\{\left.u\right|u\in L^2\left(\Omega,\mathbb{C}\right),\;u=0\;\mathrm{in}\;\Omega^c,\;\frac{\left|u(x)-u(y)\right|}{\left|x-y\right|^{\frac N2+s}}\in L^2(Q)\right\},
 \end{equation*}
endowed with the norm
\begin{equation*}
{\left\|u\right\|}_{s,2,0}=\left(\left\|u\right\|_2^2+\left[u\right]_{s,2}^2\right)^\frac12,
\end{equation*}
where ${\left[u\right]}_{s,2}=\left(\int_Q\frac{\left|u\left(x\right)-u\left(y\right)\right|^2}{\left|x-y\right|^{2s}}\mathrm d\mu\right)^\frac12$. It is easily known that $W_0^{s,2}$ is a uniformly convex reflexive Banach space and its norm is equivalent to $\left[u\right]_{s,2}$, and $H_0^s\left(\Omega,\mathbb{C}\right)=W_0^{s,2}\left(\Omega,\mathbb{C}\right)\left(k=0,1,\cdots\right)$. Further, let $G$ be an $N$-function, we define the following Lebesgue-Orlicz space
\begin{equation*}
  L^G\left(\Omega,\mathbb{C}\right):=\left\{u:\Omega\rightarrow\mathbb{C},u\;\mathrm{is}\;\mathrm{measurable}\;\mathrm{and}\;\rho_ G\left(u\right)<\infty\right\},
\end{equation*}
where $\rho_G\left(u\right):=\int_{\Omega}G\left(\left|u\left(x\right)\right|\right)\mathrm dx$.

The above Lebesgue-Orlicz spaces are equipped with the following Luxemburg norm
\begin{equation*}
{\left\|u\right\|}_G:=\inf\left\{\lambda>0\vert\rho_G\left(\frac u\lambda\right)\leq1\right\}.
\end{equation*}
We introduce
\begin{equation*}
  D_su\left(x,y\right)=\frac{u\left(x\right)-u\left(y\right)}{\left|x-y\right|^s}.
\end{equation*}
Correspondingly, define the following fractional Orlicz-Sobolev space
\begin{equation*}
  W_0^{s,G}\left(\Omega,\mathbb{C}\right):=\left\{u\in L^G\left(\Omega,\mathbb{C}\right),u=0\;\mathrm{in}\;\Omega^c,\rho_{s,G}\left(u\right)<\infty\right\},
\end{equation*}
where $s\in\left(0,1\right)$, $\rho_{s,G}\left(u\right):=\int_{Q}G\left(\left|D_su\left(x,y\right)\right|\right)\mathrm d\mu$.

The space $W_0^{s,G}\left(\Omega,\mathbb{C}\right)$ can be equipped with the the following norm
\begin{equation*}
  {\left\|u\right\|}_{s,G,0}:={\left\|u\right\|}_G+{\left[u\right]}_{s,G},
\end{equation*}
where ${\left[u\right]}_{s,G}:=\inf\left\{\left.\lambda>0\right|\rho_{s,G}\left(\frac u\lambda\right)\leq1\right\}.$

\subsection{Magnetic Fractional Orlicz-Sobolev spaces}
Some necessary definitions of Magnetic spaces in the following  are described below.

Let $A:\mathbb{R}^N\rightarrow\mathbb{R}^N$ be a smooth vector field, we consider the  magnetic fractional Orlicz-Sobolev space defined as
\begin{equation*}
W_{A,0}^{s,G}\left(\Omega,\mathbb{C}\right):=\left\{\left.u\right|u\in L^G\left(\Omega,\mathbb{C}\right),u=0\;\mathrm{in}\;\Omega^c,\rho_{s,G}^{A}\left(u\right)<\infty\right\},
\end{equation*}
where $\rho_{s,G}^A\left(u\right):=\int_{Q}G\left(\left|D_s^Au\left(x,y\right)\right|\right)\mathrm d\mu$.

The space $W_{A,0}^{s,G}\left(\Omega,\mathbb{C}\right)$ is Banach space when equipped with the following  Luxemburg norm
\begin{equation*}
\left\|u\right\|_{s,G,0}^A:={\left\|u\right\|}_G+\left[u\right]_{s,G}^A,
\end{equation*}
where ${\left[u\right]}_{s,G}^A:=\inf\left\{\left.\lambda>0\right|\rho_{s,G}^A\left(\frac u\lambda\right)\leq1\right\}$.

\begin{remark}\label{dual space}
Obviously, if we define the (topological) dual space of $W_{A,0}^{s,G}\left(\Omega,\mathbb{C}\right)$ by $W_{A,0}^{-s,G^\ast}\left(\Omega,\mathbb{C}\right)$, then $L^{G^\ast}\left(\Omega,\mathbb{C}\right)\subset W_{A,0}^{-s,G^\ast}\left(\Omega,\mathbb{C}\right)$.
\end{remark}

\subsection{Embeddings and inequalities}
There are some embeddings and inequalities that can be referenced in \cite{Alberico,Bahrouni,Bahrouni2,Bonder,Bonder2,Nezza,Pablo}.

\begin{theorem}[\cite{Nezza}]\label{embedding0}
The following embedding is continuous:
\begin{equation*}
W_{0}^{s,2}\left(\Omega\right)\hookrightarrow L^{q}\left(\Omega\right)
\end{equation*}
for any $q\in\left[1,\frac{2n}{n-2s}\right]$.
\end{theorem}

\begin{theorem}[\cite{Bahrouni}]\label{embedding1}
Let $G$ be an $N$-function. If $\left(H_1\right)$ and $\left(H_2\right)$ hold, then the embedding
\begin{equation*}
  W^{s,G}_0\left(\Omega \right)\hookrightarrow L^{G^\ast}\left(\Omega \right)
\end{equation*}
is continuous. Moreover, for any $N$-function $B$ such that $B\ll G^\ast$(see Definition \ref{$A$and$B$}), the embedding
\begin{equation*}
  W^{s,G}_0\left(\Omega \right)\hookrightarrow L^B\left(\Omega \right)
\end{equation*}
is compact.
\end{theorem}

\begin{theorem}[\cite{Pablo}]\label{embedding2}
The following embedding
\begin{equation*}
W_{A,0}^{s,G}\left(\Omega,\mathbb{C}\right)\hookrightarrow W_{0}^{s,G}\left(\Omega\right)
\end{equation*}
is continuous, i.e., for $u\in W_{A,0}^{s,G}\left(\Omega,\mathbb{C}\right)$, we have $\left|u\right|\in W_{0}^{s,G}\left(\Omega\right)$ and ${\left\|\left|u\right|\right\|}_{s,G,0}\leq\left\|u\right\|_{s,G,0}^A$.
\end{theorem}

\begin{remark}\label{remark $A$and$B$}
Let $B\left(\left|u\right|\right)=\left|u\right|^{p+1}$, by Definition \ref{$A$and$B$} and (\ref{Sobolev conjugate}), it is easy to know that $\underset{\left|u\right|\rightarrow+\infty}\lim\frac{B\left(k\left|u\right|\right)}{G^\ast\left(\left|u\right|\right)}=0$, i.e., $B\ll G^\ast$. Combining Theorems \ref{embedding1} and \ref{embedding2}, we have that the embedding  $W^{s,G}_{A,0}\left(\Omega\right)\hookrightarrow L^{p+1}\left(\Omega\right)$ is compact.

In addition, there is important Poincar$\acute{e}$ inequality for the  magnetic fractional Orlicz-Sobolev space (see \cite[Corollary 5.7]{Bonder2}), that is, for every $u\in W_{A,0}^{s,G}\left(\Omega,\mathbb{C}\right)$, there exists a constant $C$ such
that
\begin{equation*}
  {\left\|u\right\|}_G\leq C\left[u\right]_{s,G}^A,
\end{equation*}
which implies $\left\|u\right\|_{s,G,0}^A$ is equivalent to $\left[u\right]_{s,G}^{A}$.
\end{remark}

\begin{lemma}[\cite{Bahrouni,Bahrouni2,Pablo}]\label{inequality1}
Let $G$ be an $N$-function satisfying (\ref{2.1.4}) and $\zeta^\pm:\left[0,\infty\right)\rightarrow\mathbb{R}$ be defined as follows
\begin{equation*}
\zeta^+\left(t\right):=\max\left\{t^{q^-},t^{q^+}\right\},\;\zeta^-\left(t\right):=\min\left\{t^{q^-},t^{q^+}\right\}.
\end{equation*}
Then,

$\left(i\right)\zeta^-\left({\left\|u\right\|}_G\right)\leq\rho_G\left(u\right)\leq\zeta^+\left({\left\|u\right\|}_G\right);$

$\left(ii\right)\zeta^-\left({\left\|u\right\|}_{\widetilde G}\right)\leq\rho_{\widetilde G}\left(u\right)\leq\zeta^+\left({\left\|u\right\|}_{\widetilde G}\right);$

$\left(iii\right)\zeta^-\left({\left[u\right]}_{s,G}\right)\leq\rho_{s,G}\left(u\right)\leq\zeta^+\left({\left[u\right]}_{s,G}\right);$

$\left(iv\right)\zeta^-\left(\left[u\right]_{s,G}^A\right)\leq\rho_{s,G}^A\left(u\right)\leq\zeta^+\left(\left[u\right]_{s,G}^A\right).$
\end{lemma}

\begin{lemma}[\cite{Bonder}]\label{inequality2}
Assume that $G$ is an $N$-function and satisfies (\ref{2.1.1(1)}), then
\begin{equation}\label{N-function inequality}
\widetilde G\left(g\left(t\right)\right)\leq\left(q^+-1\right)G\left(t\right).
\end{equation}
\end{lemma}

\begin{lemma}[Holder's inequality\cite{Pablo}]\label{inequality3}
Let $M=M\left(x,y\right)\in L_\mu^G\left(Q\right)$ and $F=F\left(x,y\right)\in L_\mu^{\widetilde G}\left(Q\right)$, then
\begin{equation*}
\int_QM\left(x,y\right)F\left(x,y\right)\mathrm d\mu\leq2{\left\|M\right\|}_{G}{\left\|F\right\|}_{\widetilde G},
\end{equation*}
\end{lemma}
where $L_\mu^G\left(Q\right):=\left\{M\left(x,y\right):\Omega\rightarrow\mathbb{R},\;M\;\mathrm{is}\;\mathrm{measurable}\;\mathrm{and}\;\int_QG\left(\left|M\left(x,y\right)\right|\right)\mathrm d\mu<\infty\right\}$,\\$L_\mu^{\widetilde G}\left(Q\right):=\left\{F\left(x,y\right):\Omega\rightarrow\mathbb{R},\;F\;\mathrm{is}\;\mathrm{measurable}\;\mathrm{and}\;\int_Q{\widetilde G}\left(\left|F\left(x,y\right)\right|\right)\mathrm d\mu<\infty\right\}$.

\subsection{The theory of potential wells}

Now, we denote the energy functional $J\left(u\right)$  of problem (\ref{1.1}) by
\begin{equation*}
J\left(u\right):=\rho_{s,G}^A\left(u\right)-\frac1{p+1}\int_\Omega\left|u\right|^{p+1}\mathrm dx,
\end{equation*}
then the corresponding Nehari functional is
\begin{equation*}
I\left(u\right):=\int_Qg\left(\left|D_s^Au\left(x,y\right)\right|\right)\left|D_s^Au\left(x,y\right)\right|\mathrm d\mu-\int_\Omega\left|u\right|^{p+1}\mathrm dx,
\end{equation*}
and the Nehari manifold is
\begin{equation*}
{\cal N}: = \left\{ {u \in W_{A,0}^{s,G}\left( { \Omega,\mathbb{C}} \right)\backslash \left\{ 0 \right\}:I\left( u \right) = 0} \right\}.
\end{equation*}
It follows from the definitions of $J\left(u\right)$ and $I\left(u\right)$, we introduce sets
\begin{equation*}
\begin{split}
&W:=\left\{\left.u\in W_{A,0}^{s,G}\left(\Omega,\mathbb{C}\right)\right|J\left(u\right)<d,I\left(u\right)>0\right\}\cup\left\{0\right\},\\
&V:=\left\{\left.u\in W_{A,0}^{s,G}\left(\Omega,\mathbb{C}\right)\right|J\left(u\right)<d,I\left(u\right)<0\right\},
\end{split}
\end{equation*}
where $d$ is the depth of the potential well, which can be defined as
\begin{equation*}
d:=\underset{u\in\mathcal N}{\inf}J\left(u\right).
\end{equation*}

For $\delta>0$, we further define
\begin{equation}\label{C define}
\begin{split}
&I_\delta\left(u\right):=\delta\int_Qg\left(\left|D_s^Au\left(x,y\right)\right|\right)\left|D_s^Au\left(x,y\right)\right|\mathrm d\mu-\int_\Omega\left|u\right|^{p+1}\mathrm dx,\\
&{\cal N}_\delta:=\left\{ {u \in W_{A,0}^{s,G}\left( {\Omega ,\mathbb{C}} \right)\backslash \left\{ 0 \right\}:{I_\delta }\left( u \right) = 0} \right\},\\
&W_\delta:=\left\{\left.u\in W_{A,0}^{s,G}\left( {\Omega ,\mathbb{C}} \right)\right|J\left(u\right)<d\left(\delta\right),I_\delta\left(u\right)>0\right\}\cup\left\{0\right\},\\
&V_\delta:=\left\{\left.u\in W_{A,0}^{s,G}\left( {\Omega ,\mathbb{C}} \right)\right|J\left(u\right)<d\left(\delta\right),I_\delta\left(u\right)<0\right\},\\
&d\left( \delta  \right): = \mathop {\inf }\limits_{u \in {{\cal N}_\delta }} J\left( u \right).
\end{split}
\end{equation}

Next, we define some sets and functionals for the weak solution with high energy level  as follows
\begin{equation*}
\begin{split}
&{\mathcal N}_+:=\left\{u\in W_{A,0}^{s,G}\left(\Omega,\mathbb{C}\right)\left|I\left(u\right)>0\right.\right\},\\
&{\mathcal N}_-:=\left\{u\in W_{A,0}^{s,G}\left(\Omega,\mathbb{C}\right)\left|I\left(u\right)<0\right.\right\},\\
&J^\alpha:=\left\{u\in W_{A,0}^{s,G}\left(\Omega,\mathbb{C}\right)\left|J\left(u\right)\leq\alpha\right.\right\},\;\;\;\forall\alpha>d.
\end{split}
\end{equation*}
From the definitions of $\mathcal N,\;J^\alpha,\; J\left(u\right)$ and $d$, we get
\begin{equation*}
  \mathcal N^\alpha:=\mathcal N\cap J^\alpha=\left\{u\in\mathcal N\left|J\left(u\right)\leq\alpha\right.\right\},\;\;\;\forall\alpha>d.
\end{equation*}
Furthermore, define
\begin{equation*}
 \lambda_\alpha=\inf\left\{\left\|u\right\|_{H^s}^2\left|u\in\right.\mathcal N^\alpha\right\}
\end{equation*}
for $\alpha>d$. Obviously, $\lambda_\alpha$  is non-increasing.

 Let $\left\langle\cdot,\cdot\right\rangle$ the dual product between $W_{A,0}^{s,G}\left(\Omega,\mathbb{C}\right)$ and its dual space $W_{A,0}^{-s,G^\ast}\left(\Omega,\mathbb{C}\right)$, then
\begin{equation}\label{2.22}
\left\langle J'\left(\psi\right),\psi\right\rangle=I\left(\psi\right).
\end{equation}

By (\ref{1.2}), similar to \cite{Cao2}, we have
\begin{equation*}
\begin{split}
&\Phi=\left\{\mathrm{the}\;\mathrm{solutions}\;\mathrm{of}\;(\ref{1.7})\right\}\\
&\;\;\;=\left\{\psi\in W_{A,0}^{s,G}\left(\Omega,\mathbb{C}\right):\;J'\left(\psi\right)=0\;\mathrm{in}\;W_{A,0}^{-s,G^\ast}\left(\Omega,\mathbb{C}\right)\right\}\\
&\;\;\;=\left\{\psi\in W_{A,0}^{s,G}\left(\Omega,\mathbb{C}\right):\;\left\langle J'\left(\psi\right),\varphi\right\rangle=0,\;\forall\varphi\in W_{A,0}^{s,G}\left(\Omega,\mathbb{C}\right)\right\}.
\end{split}
\end{equation*}

It is easy to known the following equalities by the symmetry:
\begin{equation*}
\begin{split}
&\int_{\mathbb{R}^N}\left(-\Delta\right)^su\left(x\right)\overline{\varphi\left(x\right)}\mathrm dx=2\int_{\mathbb{R}^N}\int_{\mathbb{R}^N}\frac{u\left(x\right)-u\left(y\right)}{\left|x-y\right|^{N+2s}}\mathrm dy\overline{\varphi\left(x\right)}\mathrm dx\\
&\;\;\;\;\;\;\;\;\;\;\;\;\;\;\;\;\;\;\;\;\;\;\;\;\;\;\;\;\;\;\;\;\;\;\;\;\;=\int_{\mathbb{R}^N}\int_{\mathbb{R}^N}D_su\left(x,y\right)\overline{D_s\varphi\left(x,y\right)}\mathrm d\mu,
\end{split}
\end{equation*}
\begin{equation*}
\begin{split}
&\int_{\mathbb{R}^N}\left(-\Delta_g^A\right)^su\left(x\right)\overline{\varphi\left(x\right)}\mathrm dx=2\int_{\mathbb{R}^N}\int_{\mathbb{R}^N}\frac{g\left(\left|D_s^Au\left(x,y\right)\right|\right)}{\left|D_s^Au\left(x,y\right)\right|}D_s^Au\left(x,y\right)\frac{\mathrm dy}{\left|x-y\right|^{N+s}}\overline{\varphi\left(x\right)}\mathrm dx\\
&\;\;\;\;\;\;\;\;\;\;\;\;\;\;\;\;\;\;\;\;\;\;\;\;\;\;\;\;\;\;\;\;\;\;\;\;\;\;\;\;\;\;=\int_{\mathbb{R}^N}\int_{\mathbb{R}^N}\frac{g\left(\left|D_s^Au\left(x,y\right)\right|\right)}{\left|D_s^Au\left(x,y\right)\right|}D_s^Au\left(x,y\right)\overline{D_s^A\varphi\left(x,y\right)}\mathrm d\mu
\end{split}
\end{equation*}
for $\varphi\in W_{A,0}^{s,G}\left(\mathbb{R}^N,\mathbb{C}\right)$.

Now, we give the definition of weak solutions.

\begin{definition}[Weak solutions]\label{Weak solution}
We say that $u=u\left(x,t\right)\in L^\infty\left(0,T;W_{A,0}^{s,G}\left(\Omega,\mathbb{C}\right)\right.$ $\left.\cap H_{0}^s\left(\Omega,\mathbb{C}\right)\right)$ with $u_t\in L^2\left(0,T; H_{0}^s\left(\Omega,\mathbb{C}\right)\right)$ is a weak solution of problem (\ref{1.1}), if $u\left(x,0\right)=u_0\in W_{A,0}^{s,G}\left(\Omega,\mathbb{C}\right)\cap H_{0}^s\left(\Omega,\mathbb{C}\right)$ and
\begin{equation}\label{Weak solution equation}
\begin{split}
&\Re \left[\int_\Omega u_{\mathrm t}\overline\varphi\mathrm dx+\int_QD_s^{}u_t\left(x,y\right)\overline{D_s^{}\varphi\left(x,y\right)}\mathrm d\mu+\int_Q\frac{g\left(\left|D_s^Au\left(x,y\right)\right|\right)}{\left|D_s^Au\left(x,y\right)\right|}D_s^Au\left(x,y\right)\overline{D_s^A\varphi\left(x,y\right)}\mathrm d\mu\right]\\
&= \Re \left[ {\int_\Omega  {{\left| u \right|}^{p - 1}}u\overline \varphi \mathrm  dx  } \right],
\end{split}
\end{equation}
for any $\varphi\in W_{A,0}^{s,G}\left(\Omega,\mathbb{C}\right)\cap H_{0}^s\left(\Omega,\mathbb{C}\right)$.

Moreover, the following equality
\begin{equation}\label{Weak solution equation2}
  \int_0^t\left\|u_\tau\right\|_{s,2,0}^2\mathrm d\tau+J\left(u\right)=J\left(u_0\right)
\end{equation}
holds for $t\in\lbrack0,T)$.
\end{definition}

\begin{remark}\label{remark Weak}
Here are some derivation calculations of $J\left(u\right)$,
\begin{equation*}
\begin{split}
&\frac d{dt}\rho_{s,G}^A\left(u\right)=\Re\left[\int_Q\frac{g\left(\left|D_s^Au\left(x,y\right)\right|\right)}{\left|D_s^Au\left(x,y\right)\right|}D_s^Au\left(x,y\right)\overline{D_s^Au_t\left(x,y\right)}\mathrm d\mu\right],\\
&\frac d{dt}\left(\frac1{p+1}\int_\Omega\left|u\right|^{p+1}\mathrm dx\right)=\Re\left[\int_\Omega\left|u\right|^{p-1}u\overline{u_t}\mathrm dx\right].
\end{split}
\end{equation*}
\end{remark}

\begin{definition}[Maximal existence time]\label{Maximal existence time}
Let $u\left(t\right)$ be a weak solution of problem (\ref{1.1}). Then, for maximal existence time $T$ of $u\left(t\right)$, we have the following inclusion

\emph{(i)} If $u\left(t\right)$ exists for $t\in\left[0,\infty\right)$, then $T=+\infty$;

\emph{(ii)} If there exists a $t_0\in\left(0,\infty\right)$ such that $u\left(t\right)$ exists for $t\in\left[0,t_0\right)$,  but doesn't exist at $t=t_0$, then $T=t_0$.
\end{definition}

\begin{lemma}[\cite{Bonder2,HLieb}]\label{dense}
Let the Orlicz function $G$ satisfy the $\Delta_2$ condition, then $C_0^\infty\left(\Omega\right)$ of smooth functions with compact support is dense in $W_{A,0}^{s,G}\left(\Omega\right)$.
\end{lemma}
\begin{proof} The proof is similar to the proof process of \cite[Theorem 7.22]{HLieb} with the obvious modifications and the $\Delta_2$ condition needs to be used.
\end{proof}

From Remark \ref{remark $A$and$B$}, we can define
\begin{equation}\label{inequality2}
\frac1{C_\ast}:=\inf_{u\in W_{A,0}^{s,G}\left(\Omega,\mathbb{C}\right),u\neq0}\frac{\left[u\right]_{s,G}^A}{{\left\|u\right\|}_{p+1}}.
\end{equation}
According to the above definitions, we will give the following lemmas and proposition.

\begin{lemma}[Relations between $I_\delta\left(u\right)$ and $\int_Q g \left( {\left| {D_s^Au\left( {x,y} \right)} \right|} \right)\left| {D_s^Au\left( {x,y} \right)} \right|\mathrm d\mu $]\label{Relations between I and U}
For any $u\in W_{A,0}^{s,G}\left(\Omega,\mathbb{C}\right)$, we have

\emph{(i)} if ${0 < \int_Q g \left( {\left| {D_s^Au\left( {x,y} \right)} \right|} \right)\left| {D_s^Au\left( {x,y} \right)} \right|\mathrm d\mu  < h\left( \delta  \right)}$, then $I_\delta\left(u\right)>0$;

\emph{(ii)} if $I_\delta\left(u\right)<0$, then $\int_Q g \left( {\left| {D_s^Au\left( {x,y} \right)} \right|} \right)\left| {D_s^Au\left( {x,y} \right)} \right|\mathrm d\mu>h\left(\delta\right)$;

\emph{(iii)} if $I_\delta\left(u\right)=0$ and $\int_Q g \left( {\left| {D_s^Au\left( {x,y} \right)} \right|} \right)\left| {D_s^Au\left( {x,y} \right)} \right|\mathrm d\mu\neq0$, then $\int_Q g \left( {\left| {D_s^Au\left( {x,y} \right)} \right|} \right)$ $\left| {D_s^Au\left( {x,y} \right)} \right|\mathrm d\mu\geq h\left(\delta\right),$
where $h\left(\delta\right)$ is a function of $\delta$ and satisfies
\begin{equation*}
h\left(\delta\right)=\min\left\{\left(\frac{\left(q^-\right)^\frac{p+1}{q^-}\delta}{C_\ast^{p+1}}\right)^\frac{q^-}{p+1-q^-},\left(\frac{\left(q^-\right)^\frac{p+1}{q^-}\delta}{C_\ast^{p+1}}\right)^\frac{q^+}{p+1-q^+}\right\}.
\end{equation*}
\end{lemma}
\begin{proof}
(i) From ${0 < \int_Q g \left( {\left| {D_s^Au\left( {x,y} \right)} \right|} \right)\left| {D_s^Au\left( {x,y} \right)} \right|{\rm{d}}\mu  < h\left( \delta  \right)}$ and (\ref{2.1.4}), we have
{\small\begin{equation*}
\begin{split}
&\int_\Omega  {\left| u \right|}^{p + 1}\mathrm dx=\left\| u \right\|_{p + 1}^{p + 1} \le {C_ * }^{p + 1}{\left( {\left[ {\rm{u}} \right]_{s,G}^A} \right)}^{p + 1}\\
&\le{C_*}^{p+1}\max \left\{{\left( {\int_Q G \left( {\left| {D_s^Au\left( {x,y} \right)} \right|} \right)\mathrm d\mu } \right)}^{\frac{p + 1}{q^ + }},{\left( {\int_Q G \left( {\left| {D_s^Au\left( {x,y} \right)} \right|} \right)\mathrm d\mu } \right)}^{\frac{{p + 1}}{q^ - }} \right\}\\
&\le \frac{{C_ * }^{p + 1}}{\left( {{q^ - }} \right)^{\frac{p + 1}{q^ - }}}\int_Q g \left( {\left| {D_s^Au\left( {x,y} \right)} \right|} \right)\left| {D_s^Au\left( {x,y} \right)} \right|\mathrm d\mu \max \left\{ {\left( {\int_Q g \left( {\left| {D_s^Au\left( {x,y} \right)} \right|} \right)\left| {D_s^Au\left( {x,y} \right)} \right|\mathrm d\mu } \right)}^{\frac{p+1-q^+}{q^+}}, \right.\\
&\;\;\;\left. {\left( {\int_Q g \left( {\left| {D_s^Au\left( {x,y} \right)} \right|} \right)\left| {D_s^Au\left( {x,y} \right)} \right|\mathrm d\mu } \right)}^{\frac{p + 1 - {q^ - }}{q^ - }} \right\}\\
&< \frac{{C_ * }^{p + 1}}{\left( {{q^ - }} \right)^{\frac{p + 1}{q^ - }}}\int_Q g \left( {\left| {D_s^Au\left( {x,y} \right)} \right|} \right)\left| {D_s^Au\left( {x,y} \right)} \right|\mathrm d\mu \max \left\{ {{{\left( {h\left( \delta  \right)} \right)}^{\frac{p + 1 - {q^ + }}{q^ + }}},{\left( {h\left( \delta  \right)} \right)}^{\frac{p + 1 - {q^ - }}{q^ - }}} \right\}\\
& = \delta  {\int_Q g \left( {\left| {D_s^Au\left( {x,y} \right)} \right|} \right)\left| {D_s^Au\left( {x,y} \right)} \right|\mathrm d\mu  } ,
\end{split}
\end{equation*}}
i.e., $I_\delta\left(u\right)>0$.

(ii)  By $I_\delta\left(u\right)<0$, we know that $\int_Q g \left( {\left| {D_s^Au\left( {x,y} \right)} \right|} \right)\left| {D_s^Au\left( {x,y} \right)} \right|\mathrm d\mu\neq0$. Then,
\small\begin{equation*}
\begin{split}
&\delta\int_Qg\left(\left|D_s^Au\left(x,y\right)\right|\right)\left|D_s^Au\left(x,y\right)\right|\mathrm d\mu< \int_\Omega  {{\left| u \right|}^{p + 1}}\mathrm dx \\
&\leq \frac{{C_ * }^{p + 1}}{\left( {{q^ - }} \right)^{\frac{p + 1}{q^ - }}}\int_Q g \left( {\left| {D_s^Au\left( {x,y} \right)} \right|} \right)\left| {D_s^Au\left( {x,y} \right)} \right|\mathrm d\mu\max \left\{ {\left( {\int_Q g \left( {\left| {D_s^Au\left( {x,y} \right)} \right|} \right)\left| {D_s^Au\left( {x,y} \right)} \right|\mathrm d\mu } \right)}^{\frac{p + 1 - {q^ + }}{q^ + }},\right.  \\
&\left.{{\left( {\int_Q g \left( {\left| {D_s^Au\left( {x,y} \right)} \right|} \right)\left| {D_s^Au\left( {x,y} \right)} \right|\mathrm d\mu } \right)}^{\frac{p + 1 - {q^ - }}{q^ - }}} \right\},
\end{split}
\end{equation*}
which implies $\int_Q g \left( {\left| {D_s^Au\left( {x,y} \right)} \right|} \right)\left| {D_s^Au\left( {x,y} \right)} \right|\mathrm d\mu>h\left(\delta\right)$.

(iii) If $I_\delta\left(u\right)=0$, $\int_Q g \left( {\left| {D_s^Au\left( {x,y} \right)} \right|} \right)\left| {D_s^Au\left( {x,y} \right)} \right|\mathrm \mu\neq0$,  we obtain
\small\begin{equation*}
\begin{split}
&\delta\int_Q g \left( {\left| {D_s^Au\left( {x,y} \right)} \right|} \right)\left| {D_s^Au\left( {x,y} \right)} \right|\mathrm d\mu= \int_\Omega  {{{\left| u \right|}^{p + 1}}}\mathrm dx\\
&\leq\frac{{C_ * }^{p + 1}}{\left( {{q^ - }} \right)^{\frac{p + 1}{q^ - }}}\int_Q g \left( {\left| {D_s^Au\left( {x,y} \right)} \right|} \right)\left| {D_s^Au\left( {x,y} \right)} \right|\mathrm d\mu\max \left\{ {{\left( {\int_Q g \left( {\left| {D_s^Au\left( {x,y} \right)} \right|} \right)\left| {D_s^Au\left( {x,y} \right)} \right|{\rm{d}}\mu } \right)}^{\frac{p + 1 - {q^ + }}{q^ + }}}, \right. \\
&\left.{\left( {\int_Q g \left( {\left| {D_s^Au\left( {x,y} \right)} \right|} \right)\left| {D_s^Au\left( {x,y} \right)} \right|\mathrm d\mu } \right)}^{\frac{p + 1 - {q^ - }}{q^ - }} \right\},
\end{split}
\end{equation*}
which implies $\int_Q g \left( {\left| {D_s^Au\left( {x,y} \right)} \right|} \right)\left| {D_s^Au\left( {x,y} \right)} \right|\mathrm d\mu \ge h\left(\delta\right)$.
\end{proof}

\begin{lemma}[Properties of $J(\lambda u)$]\label{Properties of J}
Assume that $u\in W_{A,0}^{s,G}\left(\Omega,\mathbb{C}\right)$ and $\int_\Omega  {{{\left| u \right|}^{p + 1}}}\mathrm dx\neq0$, there exists a constant $\beta\in\left[q^--1,q^+-1\right]$, such that

\emph{(i)} $\underset{\lambda\rightarrow0}{\mathrm{lim}}J\left(\lambda u\right)=0,\;\underset{\lambda\rightarrow+\infty}{\mathrm{lim}}J\left(\lambda u\right)=-\infty$. In addition, there exists a $\lambda^\ast$ with
\begin{equation*}
{\lambda ^ * }: = {\left( {\frac{\int_Q {g\left( {\left| {D_s^Au\left( {x,y} \right)} \right|} \right) \left| {D_s^Au\left( {x,y} \right)} \right|}\mathrm d\mu }{\int_\Omega  {{{\left| u \right|}^{p + 1}}}\mathrm dx}} \right)^{\frac{1}{p  -\beta  }}},
\end{equation*}
such that ${\left.\frac{dJ\left(\lambda u\right)}{d\lambda}\right|}_{\lambda=\lambda^\ast}=0$, and $J\left(\lambda u\right)$ is decreasing on $\lambda\in\lbrack\lambda^\ast,\infty)$, increasing on $\lambda\in\left[0,\lambda^\ast\right]$, $\lambda=\lambda^\ast$ is the maximum point of $J\left(\lambda u\right)$;

\emph{(ii)}$I\left(\lambda^\ast u\right)=0$, $I\left(\lambda u\right)<0$ for $\lambda^\ast<\lambda<\infty$ and $I\left(\lambda u\right)>0$ for $0<\lambda<\lambda^\ast$.
\end{lemma}
\begin{proof}
From the definition of $J(u)$ and (\ref{2.1.7}), we get
 \begin{equation*}
\begin{split}
&J\left( {\lambda u} \right) = \int_Q G \left( {\lambda \left| {D_s^Au\left( {x,y} \right)} \right|} \right)\mathrm  d\mu - \frac{\lambda ^{p + 1}}{p + 1}\int_\Omega  {{\left| u \right|}^{p + 1}}\mathrm  dx\\
&\;\;\;\;\;\;\;\;\;\;\le \max \left\{ {{\lambda ^{q^ - }},{\lambda ^{q^ + }}} \right\}\int_Q G \left( {\left| {D_s^Au\left( {x,y} \right)} \right|} \right)\mathrm  d\mu- \frac{\lambda ^{p + 1}}{p + 1}\int_\Omega  {{\left| u \right|}^{p + 1}}\mathrm  dx,\\
\end{split}
 \end{equation*}
which yields $\underset{\lambda\rightarrow0}{\mathrm{lim}}J\left(\lambda u\right)=0$ and
$\underset{\lambda\rightarrow+\infty}{\mathrm{lim}}J\left(\lambda u\right)=-\infty$.
In addition, by the mean value theorem of integrals and (\ref{2.1.6}), we have
\begin{equation}\label{Properties of J1}
\begin{split}
&\frac{{\mathrm dJ\left( {\lambda u} \right)}}{\mathrm d\lambda}= \int_Q g \left( {\lambda \left| {D_s^Au\left( {x,y} \right)} \right|} \right) \left| {D_s^Au\left( {x,y} \right)} \right|\mathrm d\mu - {\lambda ^p}\int_\Omega  {{{\left| u \right|}^{p + 1}}}\mathrm  dx\\
&\;\;\;\;\;\;\;\;\;\;\;\;= {\lambda ^\beta }\int_Q {g\left( {\left| {D_s^Au\left( {x,y} \right)} \right|} \right) \left| {D_s^Au\left( {x,y} \right)} \right|}  \mathrm d\mu  - {\lambda ^p}\int_\Omega  {{{\left| u \right|}^{p + 1}}}\mathrm  dx\\
&\;\;\;\;\;\;\;\;\;\;\;\;= {\lambda ^\beta }\left( {\int_Q {g\left( {\left| {D_s^Au\left( {x,y} \right)} \right|} \right) \left| {D_s^Au\left( {x,y} \right)} \right|}\mathrm  d\mu  - {\lambda ^{p  - \beta  }}\int_\Omega  {{{\left| u \right|}^{p + 1}}}\mathrm  dx} \right),
\end{split}
\end{equation}
where $\beta\in\left[q^--1,q^+-1\right]$.
When ${\left.\frac{dJ\left(\lambda u\right)}{d\lambda}\right|}_{\lambda=\lambda^\ast}=0$, we obtain
\begin{equation*}
{\lambda ^ * } = {\left( {\frac{\int_Q {g\left( {\left| {D_s^Au\left( {x,y} \right)} \right|} \right) \left| {D_s^Au\left( {x,y} \right)} \right|}\mathrm d\mu  }{\int_\Omega  {{\left| u \right|}^{p + 1}} \mathrm dx}} \right)^{\frac{1}{p  - \beta  }}}.
\end{equation*}

From (\ref{Properties of J1}), we can infer that
\begin{equation*}
\frac{\mathrm dJ\left(\lambda u\right)}{\mathrm d\lambda}>0,\;\;\;\mathrm{for}\;0<\lambda<\lambda^\ast;\quad\frac{\mathrm dJ\left(\lambda u\right)}{\mathrm d\lambda}<0,\;\;\;\mathrm{for}\;\lambda^\ast<\lambda<\infty,
\end{equation*}
which implies
\begin{equation*}
I\left(\lambda u\right)=\lambda\frac{\mathrm dJ\left(\lambda u\right)}{d\lambda}=\left\{
\begin{split}
&>0,\;\;\;0<\lambda<\lambda^\ast;\\
&<0,\;\;\;\lambda^\ast<\lambda<\infty.
\end{split}\right.
\end{equation*}
\end{proof}

Let $M:=\left(\frac1{q^+}-\frac1{p+1}\right)\min\left\{h\left(1\right),1\right\}$, we have the following lemma.

\begin{lemma}[Depth $d$ of potential well]\label{Depth d}
For the depth of the potential well $d$, we have
\begin{equation*}
d\geq M.
\end{equation*}
\end{lemma}
\begin{proof}
If $u\in \mathcal N$, then $I\left(u\right)=0$ and $u\neq 0$. From the definitions of $J\left(u\right)$ and $I\left(u\right)$, combined with (\ref{2.1.4}), we have
\begin{equation*}
\int_Q g \left( {\left| {D_s^Au\left( {x,y} \right)} \right|} \right)\left| {D_s^Au\left( {x,y} \right)} \right|\mathrm d\mu =\int_\Omega  {{{\left| u \right|}^{p + 1}}}\mathrm dx
\end{equation*}
 and
\begin{equation*}
\begin{split}
&J\left( u \right) \ge \frac{1}{q^ + }\int_Q g \left( {\left| {D_s^Au\left( {x,y} \right)} \right|} \right)\left| {D_s^Au\left( {x,y} \right)} \right|\mathrm d\mu- \frac{1}{p + 1}\int_\Omega  {{\left| u \right|}^{p + 1}}\mathrm dx\\
&\;\;\;\;\;\;\;\; = \left( {\frac{1}{q^ + } - \frac{1}{p + 1}} \right)\int_Q g \left( {\left| {D_s^Au\left( {x,y} \right)} \right|} \right)\left| {D_s^Au\left( {x,y} \right)} \right|\mathrm d\mu .
\end{split}
\end{equation*}

From Lemma \ref{Relations between I and U} (iii) and (\ref{2.1.4}), we obtain
\begin{equation}\label{The relationship between d and J}
\begin{split}
&d = \inf\limits_{u \in {\cal N}} J(u)\\
&\;\;\ge \inf\limits_{u \in {\cal N}} \left[ {\left( {\frac{1}{q^ + } - \frac{1}{p + 1}} \right)\int_Q g \left( {\left| {D_s^Au\left( {x,y} \right)} \right|} \right)\left| {D_s^Au\left( {x,y} \right)} \right|{\rm{d}}\mu  } \right]\\
&\;\; \geq\left(\frac1{q^+}-\frac1{p+1}\right)h\left(1\right),
\end{split}
\end{equation}
which yields
\begin{equation*}
d\geq M.
\end{equation*}
\end{proof}

\begin{remark}
If $\frac{tg\left(t\right)}{G\left(t\right)}$ is a constant, we can define $M$ more precisely as $\left(\frac1{q^+}-\frac1{p+1}\right)h\left(1\right)$,
then $d=M$ (see \cite[Lemma 2.3]{Cheng}). In this case, $d$ is called critical initial energy and $M$ is called sharp-critical initial energy, respectively.
\end{remark}

\begin{lemma}\label{d properties}
Let $d(\delta)$ be defined as in (\ref{C define}), we have

\emph{(i)} $d(\delta)\geq\frac1{q^+}(1-\delta)h(\delta)+\frac{p+1-q^+}{q^+\left(p+1\right)}\delta h(\delta)$ for $0<\delta\leq1$. In particular, $d(1)\geq M$;

\emph{(ii)} there is a constant $b$, $b\in(1,\frac{p+1}{q^-}\rbrack$ such that $d(b)=0$, $d(\delta)$ is decreasing on $1\leq\delta\leq b$ and increasing on $0<\delta\leq1$.
\end{lemma}

\begin{proof}
(i) If $u\in \cal N_\delta$, that is, $I_\delta\left(u\right)=0$ and $\int_Q g \left( {\left| {D_s^Au\left( {x,y} \right)} \right|} \right)\left| {D_s^Au\left( {x,y} \right)} \right|{\rm{d}}\mu\neq0$. From Lemma \ref{Relations between I and U} (iii) and (\ref{2.1.4}), we have
\begin{equation*}
\begin{split}
&J\left( u \right) = \int_Q {\left( {1 - \delta } \right)G} \left( {\left| {D_s^Au\left( {x,y} \right)} \right|} \right)d\mu+ \int_Q {\delta G} \left( {\left| {D_s^Au\left( {x,y} \right)} \right|} \right)\mathrm d\mu- \frac{1}{{p + 1}}\int_\Omega  {\left| u \right|}^{p + 1}\mathrm dx\\
&\;\;\;\;\;\;\; \ge \frac{1}{q^ + }\left( {1 - \delta } \right)\int_Q g \left( {\left| {D_s^Au\left( {x,y} \right)} \right|} \right)\left| {D_s^Au\left( {x,y} \right)} \right|\mathrm d\mu+ \left( {\frac{1}{q^ + } - \frac{1}{p + 1}} \right)\int_\Omega  {\left| u \right|}^{p + 1}\mathrm dx\\
&\;\;\;\;\;\;\; \ge \frac{1}{q^ + }\left( {1 - \delta } \right)h\left( \delta  \right) + \frac{p + 1 - {q^ + }}{{q^ + }\left( {p + 1} \right)}\delta h\left( \delta  \right),
\end{split}
\end{equation*}
which  implies $d(1)\geq M$.

(ii) Let
\begin{equation*}
 \lambda\left(\delta\right)=\left(\frac{\delta\int_Qg\left(\left|D_s^Au\left(x,y\right)\right|\right)\left|D_s^Au\left(x,y\right)\right|\mathrm d\mu}{\int_\Omega\left|u\right|^{p+1}\mathrm dx}\right)^\frac1{p-\gamma},
\end{equation*}
then
\begin{equation}\label{d properties1}
\begin{split}
&\left|I_\delta\left(\lambda\left(\delta\right)u\right)\right|=\left|\delta\int_Qg\left(\left|\lambda\left(\delta\right)D_s^Au\left(x,y\right)\right|\right)\left|\lambda\left(\delta\right)D_s^Au\left(x,y\right)\right|\mathrm d\mu-\lambda\left(\delta\right)^{p+1}\int_\Omega\left|u\right|^{p+1}\mathrm dx\right|\\
&\;\;\;\;\;\;\;\;\;\;\;\;\;\;\;\;\;\;\;=\left|\delta{\lambda\left(\delta\right)}^{\gamma+1}\int_Qg\left(\left|D_s^Au\left(x,y\right)\right|\right)\left|D_s^Au\left(x,y\right)\right|\mathrm d\mu-\lambda\left(\delta\right)^{p+1}\int_\Omega\left|u\right|^{p+1}\mathrm dx\right|\\
&\;\;\;\;\;\;\;\;\;\;\;\;\;\;\;\;\;\;\;=\left|{\lambda\left(\delta\right)}^{\gamma+1}\left(\delta\int_Qg\left(\left|D_s^Au\left(x,y\right)\right|\right)\left|D_s^Au\left(x,y\right)\right|\mathrm d\mu-\lambda\left(\delta\right)^{p-\gamma}\int_\Omega\left|u\right|^{p+1}\mathrm dx\right)\right|\\
&\;\;\;\;\;\;\;\;\;\;\;\;\;\;\;\;\;\;\;=0,
\end{split}
\end{equation}
where $\gamma\in\left[q^--1,q^+-1\right]$. For $\lambda\left(\delta\right)u\in \mathcal N_\delta$, combining equation (\ref{2.1.4}), we obtain
\begin{equation*}
\begin{split}
&d\left(\delta\right)\leq J\left(\lambda\left(\delta\right)u\right)\\
&\;\;\;\;\;\;\;=\int_QG\left(\left|\lambda\left(\delta\right)D_s^Au\left(x,y\right)\right|\right)d\mu-\frac{\lambda\left(\delta\right)^{p+1}}{p+1}\int_\Omega\left|u\right|^{p+1}\mathrm dx\\
&\;\;\;\;\;\;\;\leq\left(\frac1{q^-}-\frac\delta{p+1}\right)\int_Qg\left(\left|\lambda\left(\delta\right)D_s^Au\left(x,y\right)\right|\right)\left|\lambda\left(\delta\right)D_s^Au\left(x,y\right)\right|\mathrm d\mu,\\
\end{split}
\end{equation*}
then $d\left(\delta\right)\leq0$ for $\delta=\frac{p+1}{q^-}$. Since $d\left(\delta\right)$ is continuous with respect to $\delta$, and $d\left(1\right)=d>0$, there exists an $b\in(1,\frac{p+1}{q^-}\rbrack$ satisfying $d\left(b\right)=0$.

Next, we will prove that $d\left(\delta'\right)<d\left(\delta''\right)$ for $0<\delta'<\delta''<1$ and $1<\delta''<\delta'<b$, respectively.

According to the definition of $\lambda\left(\delta\right)$ and $N$-function, for $0<\delta'<\delta''<1$, we have
\small\begin{equation}\label{d properties2}
\begin{split}
&J\left(\lambda\left(\delta''\right)u\right)-J\left(\lambda\left(\delta'\right)u\right)\\
&=\int_Q\int_0^{\lambda\left(\delta''\right)\left|D_s^Au\left(x,y\right)\right|}g\left(\tau\right)\mathrm d\tau \mathrm d\mu-\int_Q\int_0^{\lambda\left(\delta'\right)\left|D_s^Au\left(x,y\right)\right|}g\left(\tau\right)\mathrm d\tau\mathrm d\mu\\
&\;\;\;\;-\left(\frac{\lambda\left(\delta''\right)^{p+1}}{p+1}\int_\Omega\left|u\right|^{p+1}\mathrm dx-\frac{\lambda\left(\delta'\right)^{p+1}}{p+1}\int_\Omega\left|u\right|^{p+1}\mathrm dx\right)\\
&=\int_Q\int_{\lambda\left(\delta'\right)\left|D_s^Au\left(x,y\right)\right|}^{\lambda\left(\delta''\right)\left|D_s^Au\left(x,y\right)\right|}g\left(\tau\right)\mathrm d\tau \mathrm d\mu-\int_{\lambda\left(\delta'\right)}^{\lambda\left(\delta''\right)}\lambda\left(\delta\right)^p\mathrm d \lambda\int_\Omega\left|u\right|^{p+1}\mathrm dx\\
&=\int_Q\int_{\lambda\left(\delta'\right)}^{\lambda\left(\delta''\right)}g\left(\lambda\left(\delta\right)\left|D_s^Au\left(x,y\right)\right|\right)\left|D_s^Au\left(x,y\right)\right|\mathrm d\lambda \mathrm d\mu-\int_{\lambda\left(\delta'\right)}^{\lambda\left(\delta''\right)}\lambda\left(\delta\right)^p\int_\Omega\left|u\right|^{p+1}\mathrm dx\mathrm d\lambda\\
&=\int_{\lambda\left(\delta'\right)}^{\lambda\left(\delta''\right)}\left(1-\delta\right)\left(\int_Qg\left(\lambda\left(\delta\right)\left|D_s^Au\left(x,y\right)\right|\right)\left|D_s^Au\left(x,y\right)\right|\mathrm d\mu\right)\mathrm d\lambda,
\end{split}
\end{equation}
since $\lambda\left(\delta\right)$ is increasing with respect to $\delta$, we have $J\left(\lambda\left(\delta''\right)u\right)-J\left(\lambda\left(\delta'\right)u\right)>0$.

Similar to the above proof process, we can obtain $J\left(\lambda\left(\delta''\right)u\right)-J\left(\lambda\left(\delta'\right)u\right)>0$ for $1<\delta''<\delta'<b$.
\end{proof}

By Lemma \ref{d properties}, define $d_0=\underset{\delta\rightarrow0^+}{\mathrm{lim}}d(\delta)\geq0$.

\begin{lemma}\label{I invariable}
Let $d_0<J(u)<d$,  $u\in  W_{A,0}^{s,G}\left(\Omega,\mathbb{C}\right)$ and $\delta_1<1<\delta_2$, $\delta_1$ and $\delta_2$ are the two roots of equation $d(\delta)=J(u)$, then the sign of $I_\delta(u)$ is invariable in $\delta_1<\delta<\delta_2$.
\end{lemma}

Similar to \cite[Proposition 2.1]{Cheng} and \cite[Proposition 3.7]{Cao}, we have the following proposition.

\begin{proposition}\label{proposition}
Suppose that $u$ is the weak solution of (\ref{1.1}), $u_0\in W_{A,0}^{s,G}\left(\Omega,\mathbb{C}\right)$ and $J\left(u_0\right)=\sigma$. Then we have

 \emph{(i)} if $0<\sigma\leq d_0$, there exists a unique $\overline\delta\in\left(1,b\right)$ that satisfies $d\left(\overline\delta\right)=\sigma$, where $b$ is the constant in Lemma \ref{d properties} (ii). Furthermore, if $I\left(u_0\right)>0$, then for any $1\leq\delta<\overline\delta$, there is $u\in W_\delta$. Otherwise, if $I\left(u_0\right)<0$, then for any $1\leq\delta<\overline\delta$, there is $u\in V_\delta$;

 \emph{(ii)} if $d_0<\sigma<d$, then $\delta_1$ and $\delta_2$ satisfy $\delta_1<1<\delta_2$ and $d\left(\delta_1\right)=d\left(\delta_2\right)=\sigma$. Furthermore, if $I\left(u_0\right)>0$, then for any  $\delta_1<\delta<\delta_2$, there is $u\in W_\delta$. Otherwise, if $I\left(u_0\right)<0$, then for any $\delta_1<\delta<\delta_2$, there is $u\in V_\delta$.
\end{proposition}

\begin{lemma}[Estimate of nonlinear term $\left|u\right|^{p-1}u$]\label{Estimate of nonlinear term}
Let $p$ satisfy the conditions given by (\ref{1.1}), $\left|u_1\right|+\left|u_2\right|>0$ and $u_1\neq u_2$ for any $u_1\left(x,t\right)$, $u_2\left(x,t\right)$ with $\left(x,t\right)\in\Omega\times\left[0,T\right]$, then
\begin{equation*}
\left|\left|u_1\right|^{p-1}u_1-\left|u_2\right|^{p-1}u_2\right|\leq p\left(\left|u_1\right|+\left|u_2\right|\right)^{p-1}\left|u_1-u_2\right|.
\end{equation*}
\end{lemma}

\subsection{Main results}
We provide the following five main theorems, including the existence and uniqueness of local solutions; the existence and uniqueness, asymptotic behavior of global solutions; the finite time blowup of solutions; the convergence relationship between global solutions and ground state solutions.

\begin{theorem}\label{theorem0}[Local existence and uniqueness]
Let $u_0\in W_{A,0}^{s,G}\left(\Omega,\mathbb{C}\right)\cap H_0^s\left(\Omega,\mathbb{C}\right)$ and $p\leq\frac{N+2s}{N-2s}$, then there exist $T>0$ and a unique solution
of (\ref{1.1}) over $\left [0,T\right]$.
\end{theorem}

\begin{theorem}\label{theorem1}[Global existence and uniqueness]
 Assume that $u_0\in W_{A,0}^{s,G}\left(\Omega,\mathbb{C}\right)\cap H_0^s\left(\Omega,\mathbb{C}\right)$.

 (i) If $J\left(u_0\right)<d$  and $I\left(u_0\right)>0$, then problem (\ref{1.1}) has a global weak solution $u\in L^\infty\left(0,\infty; W_{A,0}^{s,G}\left(\Omega,\mathbb{C}\right)\cap H_0^s\left(\Omega,\mathbb{C}\right)\right)$ with $u_t\in L^2\left(0,\infty; H_0^s\left(\Omega,\mathbb{C}\right)\right)$; and the weak solution is unique for $p\leq\frac{N+sq^-}{N-sq^-}$;

 (ii) If  $J\left(u_0\right)=d$ and $I(u_0)\geq0$, then (\ref{1.1}) admits a global weak solution  $u(t)\in L^\infty\left(0,\infty; W_{A,0}^{s,G}\left(\Omega,\mathbb{C}\right)\cap H_0^s\left(\Omega,\mathbb{C}\right)\right)$ with $u_t\in L^2\left(0,\infty; H_0^s\left(\Omega,\mathbb{C}\right)\right)$; and the weak solution is unique for $p\leq\frac{N+sq^-}{N-sq^-}$;

 (iii) If $J\left(u_0\right)$ is finite and $J\left(u_0\right)>d$,  $I\left(u_0\right)>0$ and $\left\|u_0\right\|_{s,2,0}^2\leq\lambda_{J\left(u_0\right)}$, then there exists a global weak solution $u\in L^\infty\left(0,\infty; W_{A,0}^{s,G}\left(\Omega,\mathbb{C}\right)\cap H_0^s\left(\Omega,\mathbb{C}\right)\right)$ with $u_t\in L^2\left(0,\infty; H_0^s\left(\Omega,\mathbb{C}\right)\right)$; and the weak solution is unique for $p\leq\frac{N+sq^-}{N-sq^-}$.
\end{theorem}

\begin{remark}\label{0<J<d}
For Theorem \ref{theorem1}, we only need to study the case of $0<J\left(u_0\right)<d$ and $I\left(u_0\right)>0$. If $0<J\left(u_0\right)<d$ and $I\left(u_0\right)=0$, it is contradict with the definition of $d$; if $J\left(u_0\right)=0$ and $I\left(u_0\right)\geq0$, from (\ref{3.1.5b}), we have $u_0\equiv0$, which implies that it is a trivial solution; if $J\left(u_0\right)<0$ and $I\left(u_0\right)\geq0$, it is contradict with (\ref{3.1.5b}).
\end{remark}

\begin{theorem}\label{theorem2}[Asymptotic behavior]
Let $u(x,t)$ be the global bounded weak solution of (\ref{1.1}).

(i) If $J\left(u_0\right)<M$, $I\left(u_0\right)>0$ and $\int_Q\left|D_s^{}u\left(x,y\right)\right|^2\mathrm d\mu\leq C\int_Qg\left(\left|D_s^Au\left(x,y\right)\right|\right)$ $\left|D_s^Au\left(x,y\right)\right|\mathrm d\mu$, then there exist  constants $\delta_1>0$ and $\delta_2>0$ such that
\begin{equation*}
\begin{split}
&\left\|u\right\|_{s,2,0}^2<\left\|u_0\right\|_{s,2,0}^2e^{-\delta_1 t},\\
&J\left(u\right)+\left\|u\right\|_{s,2,0}^2<\left(J\left(u_0\right)+\left\|u_0\right\|_{s,2,0}^2\right)e^{-\delta_2 t},\\
&{\left\|u\right\|}_{p+1}<C_\ast\left(\frac{q^+\left(p+1\right)}{q^-\left(p+1-q^+\right)}\right)^{q^-}\left(J\left(u_0\right)+\left\|u_0\right\|_{s,2,0}^2\right)^{q^-}e^{-\delta_2 q^-t},
\end{split}
\end{equation*}
where $t\in\lbrack0,\infty)$;

(ii) If $J\left(u_0\right)=M$, $I(u_0)>0$ and $\int_Q\left|D_s^{}u\left(x,y\right)\right|^2\mathrm d\mu\leq C\int_Qg\left(\left|D_s^Au\left(x,y\right)\right|\right)$ $\left|D_s^Au\left(x,y\right)\right|\mathrm d\mu$, then there exist constants $t_1>0$, $\varsigma_1>0$ and $\varsigma_2>0$ such that
\begin{equation*}
\begin{split}
&\left\|u\right\|_{s,2,0}^2<\left\|u\left(t_1\right)\right\|_{s,2,0}^2e^{-\varsigma_1\left(t-t_1\right)},\\
&J\left(u\right)+\left\|u\right\|_{s,2,0}^2<\left(J\left(u\left(t_1\right)\right)+\left\|u\left(t_1\right)\right\|_{s,2,0}^2\right)e^{-\varsigma_2\left(t-t_1\right)},\\
&{\left\|u\right\|}_{p+1}<C_\ast\left(\frac{q^+\left(p+1\right)}{q^-\left(p+1-q^+\right)}\right)^{q^-}\left(J\left(u\left(t_1\right)\right)+\left\|u\left(t_1\right)\right\|_{s,2,0}^2\right)^{q^-}e^{-\varsigma_2 q^-\left(t-t_1\right)},
\end{split}
\end{equation*}
where $t\in(t_1,+\infty)$;

(iii) If $J\left(u_0\right)$ is finite and $J\left(u_0\right)>d$,  $I\left(u_0\right)>0$ and $\left\|u_0\right\|_{s,2,0}^2\leq\lambda_{J\left(u_0\right)}$, then $u\left(t\right)\rightarrow0$ as $t\rightarrow+\infty$.
\end{theorem}

\begin{remark}
In Theorem \ref{theorem2}, we present the following examples to ensure that the assumption
$\int_Q\left|D_s^{}u\left(x,y\right)\right|^2\mathrm d\mu\leq C\int_Qg\left(\left|D_s^Au\left(x,y\right)\right|\right)\left|D_s^Au\left(x,y\right)\right|\mathrm d\mu$ holds. It should be noted that $\int_Q\left|D_s^{}u\left(x,y\right)\right|^2\mathrm d\mu\leq C\int_Q\left|D_s^Au\left(x,y\right)\right|^2\mathrm d\mu$ may exist by $\left|\left|u\left(x\right)\right|\right.$ $-\left.\left|u\left(y\right)\right|\right|\leq\left|u\left(x\right)-e^{i\left(x-y\right)A\left(\frac{x+y}2\right)}u\left(y\right)\right|$(\cite[subsection 2.5]{Pablo}).

(i) If $\int_Q\left|D_s^{}u\left(x,y\right)\right|^2\mathrm d\mu\leq C\int_Q\left|D_s^Au\left(x,y\right)\right|^2\mathrm d\mu$, $q^-\geq2$ and $\left|D_s^Au\right|>1$, by (\ref{2.1.4}) and (\ref{2.1.7}), we have
\begin{equation*}
\begin{split}
&\int_Q\left|D_s^{}u\left(x,y\right)\right|^2\mathrm d\mu\leq C\int_Q\left|D_s^Au\left(x,y\right)\right|^2\mathrm d\mu\\
&\;\;\;\;\;\;\;\;\;\;\;\;\;\;\;\;\;\;\;\;\;\;\;\;\;\;\;\;\leq C\int_Q\left|D_s^Au\left(x,y\right)\right|^{q^-}\mathrm d\mu\\
&\;\;\;\;\;\;\;\;\;\;\;\;\;\;\;\;\;\;\;\;\;\;\;\;\;\;\;\;\leq C\int_QG\left(\left|D_s^Au\left(x,y\right)\right|\right)\mathrm d\mu\\
&\;\;\;\;\;\;\;\;\;\;\;\;\;\;\;\;\;\;\;\;\;\;\;\;\;\;\;\;\leq C\int_Qg\left(\left|D_s^Au\left(x,y\right)\right|\right)\left|D_s^Au\left(x,y\right)\right|\mathrm d\mu;
\end{split}
\end{equation*}

(ii) If $\int_Q\left|D_s^{}u\left(x,y\right)\right|^2\mathrm d\mu\leq C\int_Q\left|D_s^Au\left(x,y\right)\right|^2\mathrm d\mu$, $1<q^+<2$ and $\left|D_s^Au\right|\leq1$,  from (\ref{2.1.6}), we obtain
\begin{equation*}
\begin{split}
&\int_Q\left|D_s^{}u\left(x,y\right)\right|^2\mathrm d\mu\leq C\int_Q\left|D_s^Au\left(x,y\right)\right|^2\mathrm d\mu\\
&\;\;\;\;\;\;\;\;\;\;\;\;\;\;\;\;\;\;\;\;\;\;\;\;\;\;\;\;\leq C\int_Q\left|D_s^Au\left(x,y\right)\right|^{q^+}\mathrm d\mu\\
&\;\;\;\;\;\;\;\;\;\;\;\;\;\;\;\;\;\;\;\;\;\;\;\;\;\;\;\;\leq C\int_QG\left(\left|D_s^Au\left(x,y\right)\right|\right)\mathrm d\mu\\
&\;\;\;\;\;\;\;\;\;\;\;\;\;\;\;\;\;\;\;\;\;\;\;\;\;\;\;\;\leq C\int_Qg\left(\left|D_s^Au\left(x,y\right)\right|\right)\left|D_s^Au\left(x,y\right)\right|\mathrm d\mu.
\end{split}
\end{equation*}
\end{remark}

\begin{theorem}\label{theorem3}[Blowup]
Let $u_0\in W_{A,0}^{s,G}\left(\Omega,\mathbb{C}\right)$.

(i) If $J\left(u_0\right)<M$ and $I\left(u_0\right)<0$, then the weak solution $u(t)$ of (\ref{1.1}) blows up in finite time;

(ii) If $J(u_0)=M$ and $I(u_0)<0$, then the weak solution $u(t)$ of (\ref{1.1}) blows up in finite time.
\end{theorem}

\begin{theorem}\label{theorem4}[Convergence relations]
Let $u(x,t)$ be a global solution of (\ref{1.1}), if $W_{A,0}^{s,G}\left(\Omega,\mathbb{C}\right)\cap H_0^s\left(\Omega,\mathbb{C}\right)=W_{A,0}^{s,G}\left(\Omega,\mathbb{C}\right)$, then there exist a function $u^\ast\in\Phi$ and an increasing sequence $\left\{t_k\right\}_{k=1}^\infty$ with $t_k\rightarrow+\infty$ as $k\rightarrow+\infty$ such that
 \begin{equation*}
\lim_{k\rightarrow+\infty}\left[u\left(\cdot,t_k\right)-u^\ast\right]_{s,G}^A=0.
\end{equation*}

\end{theorem}

\begin{remark}\label{ground state solution}
Taking $M\left(\rho_{s,G}^A\left(u\right)\right)=1,\;V\left(x\right)=0,\;f\left(x,\left|u\right|\right)=\left|u\right|^{p-1}$ for the equation $\left(1.3\right)$ in the reference \cite{Pablo}, similar to the proof process in the reference \cite[Theorem 2]{Pablo}, it is easy to know that the ground state solutions of (\ref{1.7}) exist.

In addition, the weak solution of the modified formula in \cite{Pablo} will be in space $W_{A,0}^{s,G}\left(\Omega,\mathbb{C}\right)$. For consistency, it is necessary to change the space of the global solution in $W_{A,0}^{s,G}\left(\Omega,\mathbb{C}\right)\cap H_0^s\left(\Omega,\mathbb{C}\right)$ to $W_{A,0}^{s,G}\left(\Omega,\mathbb{C}\right)$.
\end{remark}

\section{Local solutions}\label{Local solutions}
In this section,  we use the Galerkin method to prove the local existence and uniqueness of solutions for problem (\ref{1.1}).

\textbf{Proof of Theorem \ref{theorem0}.} The proof of the theorem is divided into two steps. First we need to prove the existence and uniqueness of weak solutions of the following problem (\ref{03}) corresponding to problem (\ref{1.1}) by Galerkin method, then from the Contraction Mapping Principle, we prove the existence and uniqueness of local solutions to problem (\ref{1.1}).

\textbf{Step I:} Consider space $\mathcal{H}=C\left( \left[ {0,T} \right]; H_0^s\left(\Omega,\mathbb{C}\right) \right)$ for every $T>0$, and define the norm on $\mathcal{H}$ as follows
\begin{equation*}
\left\|u\right\|_{\mathcal H}^2=\max_{t\in\left[0,T\right]}\int_Q\left|D_s^{}u\left(x,y\right)\right|^2\mathrm d\mu.
\end{equation*}

Next for every $T>0$ and $u \in \mathcal{H}$, we shall prove that there exists a unique $v \in \mathcal{H}$ satisfying
\begin{equation}\label{03}
\left\{ {\begin{split}
&{v_t}+{\left( { - \Delta} \right)}^s v_t + {{\left( { - \Delta _g^A} \right)}^s}v=\left|u\right|^{p-1}u,\;\;\;\;\;\;\;\;\;\;\;\;\;\;\;0<t<T,x\in\Omega,\\
&v\left(x,t\right)=0,\;\;\;\;\;\;\;\;\;\;\;\;\;\;\;\;\;\;\;\;\;\;\;\;\;\;\;\;\;\;\;\;\;\;\;\;\;\;\;\;\;\;\;\;\;0<t<T,\;x\in\mathbb{R}^N\backslash\Omega,\\
&{v\left( {x,0} \right) = {u_0}\left( x \right),\;\;\;\;\;\;\;\;\;\;\;\;\;\;\;\;\;\;\;\;\;\;\;\;\;\;\;\;\;\;\;\;\;\;\;\;\;\;\;\;\;\;\;\;\;\;\;\;\;\;\;\;\;\;\;\;\;\;\;\;\;x\in\Omega.}
\end{split}
} \right.
\end{equation}

\textbf{Existence}. Let ${\left\{\omega_j\left(x\right)\right\}}_{j\in\mathbb{N}}$ be the Laplace operator subject to the Dirichlet boundary condition:
\begin{equation*}
\left\{\begin{split}
&\left(-\Delta\right)^s\omega_j=\lambda_j\omega_j,\;x\in\Omega,\\
&\omega_j=0,\;x\in\mathbb{R}^N\backslash\Omega,\;j=1,2,\cdots,
\end{split}
\right.
\end{equation*}
 where $\lambda_j (j=1,2,\cdots)$ are the characteristic values. Then ${\left(\omega_k,\omega_j\right)}_{H_0^s\left(\Omega\right)}=\left(1+\lambda_j\right)$ ${\left(\omega_k,\omega_j\right)}_{L^2\left(\Omega\right)}$, which means that ${\left\{\omega_j\left(x\right)\right\}}_{j\in\mathbb{N}}$ are orthogonal in $W_{A,0}^{s,G}\left(\Omega,\mathbb{C}\right)\cap H_0^s\left(\Omega,\mathbb{C}\right)$ and in $L^2\left(\Omega,\mathbb{C}\right)$. Construct the approximate solution $v_m\left(x,t\right)$ of (\ref{1.1}) as follows
\begin{equation*}
v_m\left(x,t\right)=\sum_{j=1}^mg_{jm}\left(t\right)\omega_j(x),\;\;\;m=1,2,\cdots,
\end{equation*}
where $g_{jm}\left(t\right)$ satisfies
\begin{equation}\label{04}
\begin{split}
&\Re\left[\int_\Omega v_{mt}\overline{\omega_s} \mathrm dx+\int_QD_s^{}v_{mt}\left(x,y\right)\overline{D_s^{}\omega_s\left(x,y\right)}\mathrm d\mu\right.\\
&\left.+\int_Q\frac{g\left(\left|D_s^Av_m\left(x,y\right)\right|\right)}{\left|D_s^Av_m\left(x,y\right)\right|}D_s^Av_m\left(x,y\right)\overline{D_s^A\omega_s\left(x,y\right)}\mathrm d\mu\right]=\Re\left[\int_\Omega \left|u\right|^{p-1}u\overline{\omega_s}\mathrm dx\right],
\end{split}
\end{equation}
\begin{equation}\label{05}
v_m\left(x,0\right)=\sum_{j=1}^ma_{jm}(0)\omega_j(x)\rightarrow u_0(x)\;\mathrm{in}\; W_{A,0}^{s,G}\left(\Omega,\mathbb{C}\right)\cap H_0^s\left(\Omega,\mathbb{C}\right),
\end{equation}
for $1\leq s\leq m$, in which $a_{jm}(0)=g_{jm}{(0)}$.

Equations (\ref{04}) and (\ref{05}) give an initial value problem for a system of ordinary differential equations
\begin{equation}\label{06}
\left\{
\begin{split}
&\left(1+\lambda_s\right)\frac{d}{ dt}\Re\left[g_{sm}\left(t\right)\right]=\Re\left[f^\ast\left(g_{1m},g_{2m},\cdots\;,g_{mm}\left(t\right),u\left(t\right)\right)\right],\\
&g_{sm}\left(0\right)=\int_\Omega u_0\overline{\omega_s}\mathrm dx,
\end{split}
\right.
\end{equation}
where $f^\ast\left(g_{1m}\left(t\right),g_{2m}\left(t\right),\cdots\;,g_{mm}\left(t\right),u\left(t\right)\right)=-\int_Q\frac{g\left(\left|D_s^Av_m\left(x,y\right)\right|\right)}{\left|D_s^Av_m\left(x,y\right)\right|}D_s^Av_m\left(x,y\right)$\\$\overline{D_s^A\omega_s\left(x,y\right)}\mathrm d\mu+\int_\Omega\left|u\right|^{p-1}u\overline{\omega_s}\mathrm dx$ for all $j$, the initial value problem (\ref{06}) admits a local solution by  standard existence theory for ordinary differential equations.

Multiplying (\ref{04}) by $g_{jm}\left(t\right)$ and summing for $j$ from $1$ to $m$, we get
\begin{equation}\label{07}
\frac12\frac d{dt}\left\|v_m\right\|_{s,2,0}^2+\int_Qg\left(\left|D_s^Av_m\left(x,y\right)\right|\right)\left|D_s^Av_m\left(x,y\right)\right|\mathrm d\mu=\Re\left[\int_\Omega\left|u\right|^{p-1}u{\overline{v_m}}\mathrm dx\right].
\end{equation}
By H$\ddot{\mathrm o}$lder and Young's inequalities,  we can obtain
\begin{equation}\label{08}
\begin{split}
&\left|\int_\Omega\left|u\right|^{p-1}u{\overline v}_m\mathrm dx\right|\leq\int_\Omega\left|u\right|^p\left|v_m\right|\mathrm dx\\
&\;\;\;\;\;\;\;\;\;\;\;\;\;\;\;\;\;\;\;\;\;\;\;\;\;\;\;\leq\left(\int_\Omega\left|u\right|^\frac{2Np}{N+2s}\mathrm dx\right)^\frac{N+2s}{2N}\left(\int_\Omega\left|v_m\right|^\frac{2N}{N-2s}\mathrm dx\right)^\frac{N-2s}{2N}\\
&\;\;\;\;\;\;\;\;\;\;\;\;\;\;\;\;\;\;\;\;\;\;\;\;\;\;\;\leq\left\|u\right\|_\frac{2Np}{N+2s}^p{\left\|v_m\right\|}_\frac{2N}{N-2s}\\
&\;\;\;\;\;\;\;\;\;\;\;\;\;\;\;\;\;\;\;\;\;\;\;\;\;\;\;\leq C\left[v_m\right]_{s,2}\\
&\;\;\;\;\;\;\;\;\;\;\;\;\;\;\;\;\;\;\;\;\;\;\;\;\;\;\;\leq C+\frac12\left[v_m\right]_{s,2}^2.
\end{split}
\end{equation}
From (\ref{07}) and (\ref{08}), we have
\begin{equation}\label{010}
\frac d{dt}\left\|v_m\right\|_{s,2,0}^2\leq C+\left[v_m\right]_{s,2}^2.
\end{equation}
Then combining Gronwall's inequality and (\ref{010}), we have
\begin{equation}\label{010a}
\left\|v_m\right\|_{s,2,0}^2\leq e^T\left(CT+C\right).
\end{equation}

Multiplying (\ref{04}) by $g'_{jm}\left(t\right)$, summing up with respect to $j$ and integrating with respect to the time variable from $0$ to $t$, we have
\begin{equation}\label{011}
\begin{split}
&\int_0^t\left\|v_{m\tau}\right\|_{s,2,0}^2\mathrm d\tau+\rho_{s,G}^A\left(v_m\right)=\Re\left[\int_0^t\int_\Omega\left|u\right|^{p-1}u{\overline{v_{m\tau}}}\mathrm dx\mathrm d\tau\right]+\rho_{s,G}^A\left(v_m\left(0\right)\right),\;0\leq t\leq T.
\end{split}
\end{equation}
Similar to (\ref{08}),  we can get
\begin{equation}\label{012}
\left|\int_0^t\int_\Omega\left|u\right|^{p-1}u\overline{v_{m\tau}}\mathrm dx\mathrm d\tau\right|\leq CT+\frac12\int_0^t\left[v_{m\tau}\right]_{s,2}^2\mathrm d\tau.
\end{equation}
Combining (\ref{011}) and (\ref{012}) gives
\begin{equation}\label{013}
\int_0^t\left\|v_{m\tau}\right\|_{s,2,0}^2\mathrm d\tau+\rho_{s,G}^A\left(v_m\right)\leq CT.
\end{equation}

Review Lemma \ref{inequality1}, combining (\ref{N-function inequality}) and (\ref{013}), we have
\begin{equation}\label{013a}
\begin{aligned}
&{\left\|\frac{g\left(\left|D_s^Av_m\right|\right)}{\left|D_s^Av_m\right|}D_s^Av_m\right\|}_{\widetilde G}\\
&\leq\max\left\{\left(\int_Q\widetilde G\left(g\left(\left|D_s^Av_m\right|\right)\right) d\mu\right)^\frac1{q^+},\left(\int_Q\widetilde G\left(g\left(\left|D_s^Av_m\right|\right)\right)\mathrm d\mu\right)^\frac1{q^-}\right\}\\
&\leq \max\left\{\left({\left(q^+-1\right)\int_Q}G\left(\left|D_s^Av_m\right|\right)\mathrm d\mu\right)^\frac1{q^+},\left({\left(q^+-1\right)\int_QG}\left(\left|D_s^Av_m\right|\right)\mathrm d\mu\right)^\frac1{q^-}\right\}\\
&\leq C\max\left\{\left(T\right)^\frac1{q^+},\left(T\right)^\frac1{q^-}\right\}.
  \end{aligned}
\end{equation}
Hence, from (\ref{010a}) (\ref{013}) and (\ref{013a}), there exists a $v$ and a subsequence $\left\{v_m\right\}$ of $\left\{v_m\right\}$ such that
\begin{equation*}
\left\{ \begin{split}
&{v_m} \rightharpoonup v{\rm{\;in\;}}{L^\infty }\left( {\left[ {0,T} \right];W_{A,0}^{s,G}\left(\Omega,\mathbb{C}\right) \cap {H_0^s}\left( \Omega,\mathbb{C} \right)} \right){\rm{\;weak\;star,}}\\
&\frac{g\left(\left|D_s^Av_m\left(x,y\right)\right|\right)}{\left|D_s^Av_m\left(x,y\right)\right|}D_s^Av_m\left(x,y\right)\rightharpoonup\xi\;\mathrm{in}\;L^\infty\left(\left[0,T\right];L_\mu^{G^\ast}\left(\Omega,\mathbb{C}\right)\right)\;\mathrm{weak}\;\mathrm{star},\\
&{v_{mt}} \rightharpoonup {v_t}{\rm{\;in\;}}{L^2}\left( {\left[ {0,T} \right];{H_0^s}\left( \Omega,\mathbb{C}\right)} \right){\rm{\;weakly}}{\rm{.}}
\end{split}
 \right.
\end{equation*}

In (\ref{04}) we fixed $s$, letting $m  \to \infty$, we  get
\begin{equation*}
\Re\left[\int_\Omega v_t\overline{\omega_s}\mathrm dx+\int_QD_s^{}v_t\left(x,y\right)\overline{D_s^{}\omega_s\left(x,y\right)}\mathrm d\mu+\int_Q\xi\overline{D_s^A\omega_s\left(x,y\right)}\mathrm d\mu\right]=\Re\left[\int_\Omega \left|u\right|^{p-1}u\overline{\omega_s}\mathrm dx\right],
\end{equation*}
and
\begin{equation}\label{014}
\Re\left[\int_\Omega v_t\overline\varphi\mathrm dx+\int_QD_s^{}v_t\left(x,y\right)\overline{D_s^{}\varphi\left(x,y\right)}\mathrm d\mu+\int_Q\xi\overline{D_s^A\varphi\left(x,y\right)}\mathrm d\mu\right]=\Re\left[\int_\Omega \left|u\right|^{p-1}u\overline \varphi\mathrm dx\right]
\end{equation}
for any $\varphi\in W_{A,0}^{s,G}\left(\Omega,\mathbb{C}\right) \cap {H_0^s}\left( \Omega,\mathbb{C}  \right),\;\mathrm a.\mathrm e.\;t>0.$ In addition, according to Aubin-Lions lemma, we have $v \in C\left( {\left[ {0,T} \right]; {H_0^s}\left( \Omega,\mathbb{C} \right)} \right)$.

Now, we prove
\begin{equation}\label{014a}
  \Re\left[\int_Q\xi\overline{D_s^A\varphi\left(x,y\right)}\mathrm d\mu\right]=\Re\left[\int_Q\frac{g\left(\left|D_s^Av\left(x,y\right)\right|\right)}{\left|D_s^Av\left(x,y\right)\right|}D_s^Av\left(x,y\right)\overline{D_s^A\varphi\left(x,y\right)}\mathrm d\mu\right].
\end{equation}
For any $\phi\in L^\infty\left(\left[0,T\right];W_{A,0}^{s,G}\left(\Omega,\mathbb{C}\right)\cap H_0^s\left(\Omega,\mathbb{C}\right)\right)$, we have
\begin{equation*}
\begin{split}
&\Re\left[\int_Q\left(\frac{g\left(\left|D_s^Av_m\left(x,y\right)\right|\right)}{\left|D_s^Av_m\left(x,y\right)\right|}D_s^Av_m\left(x,y\right)-\frac{g\left(\left|D_s^A\phi\left(x,y\right)\right|\right)}{\left|D_s^A\phi\left(x,y\right)\right|}D_s^A\phi\left(x,y\right)\right)\right.\\
&\left.\left(\overline{D_s^Av_m\left(x,y\right)}-\overline{D_s^A\phi\left(x,y\right)}\right)\mathrm d\mu\right]\geq0.
\end{split}
\end{equation*}
Integrating equation (\ref{07}) with respect to time $t$ from $0$ to $T$, we obtain
\begin{equation*}
 \begin{split}
&  \int_0^T\int_Qg\left(\left|D_s^Av_m\left(x,y\right)\right|\right)\left|D_s^Av_m\left(x,y\right)\right|\mathrm d\mu\mathrm dt\\
&=\Re\left[\int_0^T\int_\Omega\left|u\right|^{p-1}u\overline{v_m}\mathrm dx\mathrm dt\right]-\frac12\left\|v_m\left(T\right)\right\|_{s,2,0}^2+\frac12\left\|v_m\left(0\right)\right\|_{s,2,0}^2.
 \end{split}
\end{equation*}

According to the weak lower semi-continuity of the $L^2\left(\Omega\right)$, it can be concluded that
\begin{equation}\label{014a}
\begin{split}
&\underset{ m\rightarrow\infty}{\lim\sup}\Re\left[\int_0^T\int_Q\left(g\left(\left|D_s^Av_m\left(x,y\right)\right|\right)\left|D_s^Av_m\left(x,y\right)\right|-\frac{g\left(\left|D_s^Av_m\left(x,y\right)\right|\right)}{\left|D_s^Av_m\left(x,y\right)\right|}D_s^Av_m\left(x,y\right)\right.\right.\\
&\left.\left.\overline{D_s^A\phi\left(x,y\right)}-\frac{g\left(\left|D_s^A\phi\left(x,y\right)\right|\right)}{\left|D_s^A\phi\left(x,y\right)\right|}D_s^A\phi\left(x,y\right)\left(\overline{D_s^Av_m\left(x,y\right)}-\overline{D_s^A\phi\left(x,y\right)}\right)\right)\mathrm d\mu\mathrm dt\right]\\
&\leq\Re\left[\int_0^T\int_\Omega\left|u\right|^{p-1}u\overline v\mathrm dx\mathrm dt-\frac12\left\|v\left(T\right)\right\|_{s,2,0}^2+\frac12\left\|v\left(0\right)\right\|_{s,2,0}^2-\int_0^T\int_Q\xi\overline{D_s^A\phi\left(x,y\right)}\mathrm d\mu\mathrm dt\right.\\
&\left.\;\;\;\;\;-\int_0^T\int_Q\frac{g\left(\left|D_s^A\phi\left(x,y\right)\right|\right)}{\left|D_s^A\phi\left(x,y\right)\right|}D_s^A\phi\left(x,y\right)\left(\overline{D_s^Av\left(x,y\right)}-\overline{D_s^A\phi\left(x,y\right)}\right)\mathrm d\mu\mathrm dt\right].
\end{split}
\end{equation}
Let $\varphi=v$ in (\ref{014}) and integrate equation (\ref{014}) with respect to $t$ from $0$ to $T$, then it can be rewritten as
\begin{equation}\label{014b}
  \Re\left[\int_0^T\int_Q\xi\overline{D_s^Av\left(x,y\right)}\mathrm d\mu\mathrm dt\right]=\Re\left[\int_0^T\int_\Omega \left|u\right|^{p-1}u\overline v\mathrm dx\mathrm dt\right]-\frac12\left\|v\left(T\right)\right\|_{s,2,0}^2+\frac12\left\|v\left(0\right)\right\|_{s,2,0}^2.
\end{equation}
Combining (\ref{014a})) and (\ref{014b}), it can be concluded that
\begin{equation}\label{014c}
\begin{split}
&\Re\left[\int_0^T\int_Q\xi\left(\overline{D_s^Av\left(x,y\right)}-\overline{D_s^A\phi\left(x,y\right)}\right)\mathrm d\mu\mathrm dt\right.\\
&\left.-\int_0^T\int_Q\frac{g\left(\left|D_s^A\phi\left(x,y\right)\right|\right)}{\left|D_s^A\phi\left(x,y\right)\right|}D_s^A\phi\left(x,y\right)\left(\overline{D_s^Av\left(x,y\right)}-\overline{D_s^A\phi\left(x,y\right)}\right)\mathrm d\mu\mathrm dt\right]\geq0.
\end{split}
\end{equation}
Let $\phi=v-\lambda\tau$ in (\ref{014c}), $\lambda\geq0$, then
\begin{equation*}
  \Re\left[\int_0^T\int_Q\left(\xi-\frac{g\left(\left|D_s^A\left(v-\lambda\tau\right)\left(x,y\right)\right|\right)}{\left|D_s^A\left(v-\lambda\tau\right)\left(x,y\right)\right|}D_s^A\left(v-\lambda\tau\right)\left(x,y\right)\right)\overline{D_s^A\tau\left(x,y\right)}\mathrm d\mu\mathrm dt\right]\geq0.
\end{equation*}
Let $\lambda\rightarrow0$, we have
\begin{equation*}
 \Re\left[\int_0^T\int_Q\left(\xi-\frac{g\left(\left|D_s^Av\left(x,y\right)\right|\right)}{\left|D_s^Av\left(x,y\right)\right|}D_s^Av\left(x,y\right)\right)\overline{D_s^A\tau\left(x,y\right)}\mathrm d\mu\mathrm dt\right]\geq0.
\end{equation*}
Obviously, let $\lambda\leq0$, we can derive a similar inequality using $\leq$ instead of $\geq$, then (\ref{014})  holds.

\textbf{Uniqueness}. Suppose that (\ref{03}) admits two weak solutions $\psi_1$ and $\psi_2$ with the same initial value condition. By subtracting the two equations corresponding to $\psi_1$ and $\psi_2$ respectively, and testing it with $\psi_1 - \psi_2$, we  have
\begin{equation*}
\begin{split}
&\Re\left[\int_0^t\int_\Omega\left[\left(\psi_{1\tau}-\psi_{2\tau}\right)\left(\overline{\psi_1}-\overline{\psi_2}\right)\right]\mathrm dx\mathrm d\tau+\int_0^t\int_QD_s^{}\left(\psi_{1\tau}-\psi_{2\tau}\right)D_s^{}\left(\overline{\psi_1}-\overline{\psi_2}\right)\mathrm d\mu\mathrm d\tau\right.\\
&\left.+\int_0^t\int_\Omega\left[\left(-\Delta_g^A\right)^s\psi_1-\left(-\Delta_g^A\right)^s\psi_2\right]\left(\overline{\psi_1}-\overline{\psi_2}\right)\mathrm dx\mathrm d\tau\right]=0,
\end{split}
\end{equation*}
where
\begin{equation*}
\begin{split}
&\Re\left[\int_0^t\int_\Omega\left(\left(-\Delta_g^A\right)^s\psi_1-\left(-\Delta_g^A\right)^s\psi_2\right)\left(\overline{\psi_1}-\overline{\psi_2}\right)\mathrm dx\mathrm d\tau\right]\\
&=\Re\left[\int_0^t\int_Q\left(\frac{g\left(\left|D_s^A\psi_1\right|\right)}{\left|D_s^A\psi_1\right|}D_s^A\psi_1-\frac{g\left(\left|D_s^A\psi_2\right|\right)}{\left|D_s^A\psi_2\right|}D_s^A\psi_2\right)\left(\overline{D_s^A\psi_1}-\overline{D_s^A\psi_2}\right)\mathrm d\mu\mathrm d\tau\right]\\
&\geq0.
\end{split}
\end{equation*}
Then, it is easy to know that ${\psi_1} \equiv {\psi_2}$.

\textbf{Step II:} Let $R^2=2\left\|u_0\right\|_{s,2,0}^2$, for any $T > 0$, set
\begin{equation*}
{X_T}: = \left\{ {u \in \mathcal{H}\left| {{{\left\| u \right\|}_\mathcal{H}} \le R} \right.} \right\}.
\end{equation*}
According to Step I, for any $u \in {X_T}$ and the unique solution $v \in \mathcal{H}$ of the problem (\ref{03}), we can define $v = \Lambda \left( u \right)$. Now, we prove that $\Lambda \left( {{X_T}} \right) \subset {X_T}$ is a contractive map.

(i) For given $u \in {X_T}$, Multiplying the equation of (\ref{04}) by $g_{jm}\left(t\right)$, summing up with respect to $j$,  we get
\begin{equation}\label{015}
\frac12\frac d{dt}\left\|v\right\|_{s,2,0}^2+\int_Qg\left(\left|D_s^Av\left(x,y\right)\right|\right)\left|D_s^Av\left(x,y\right)\right|\mathrm d\mu-\Re\left[\int_\Omega\left|u\right|^{p-1}u\overline v\mathrm dx\right]=0
\end{equation}

By similar estimates as (\ref{08}), we have
\begin{equation}\label{017}
\left|\int_\Omega\left|u\right|^{p-1}u\overline v\mathrm dx\right|\leq CR^{2p}+\frac 12\left[v\right]_{s,2}^2.
\end{equation}

Combining (\ref{015}) and (\ref{017}), from Gronwall's inequality, we obtain
\begin{equation*}
\left\|v\right\|_{\mathcal H}^2\leq e^T\left(\frac12R^2+CTR^{2p}\right).
\end{equation*}
Select a sufficiently small $T$ such that ${\left\| v \right\|_\mathcal{H}} \le R$, then $\Lambda \left( {{X_T}} \right) \subset {X_T}$.

(ii) Next, we prove that such map is contractive. Taking ${z_1},\;{z_2} \in {X_T}$ to be  the known functions in the right term of (\ref{03}) respectively, subtracting
the two equations in form of (\ref{03}) for ${v_1} = \Lambda \left( {{z_1}} \right)$ and ${v_2} = \Lambda \left( {{z_2}} \right)$ respectively, from (i), we know ${v_1},{v_2} \in {X_T}$. Setting $\widehat v = {v_1} - {v_2}$ and testing the both sides by $\widehat v$, we have
\begin{equation}\label{021(1)}
\begin{split}
&\Re\left[\frac12\left\|\widehat v\right\|_{s,2,0}^2+\int_0^t\int_Q^{}\left(\frac{g\left(\left|D_s^Av_1\right|\right)}{\left|D_s^Av_1\right|}D_s^Av_1-\frac{g\left(\left|D_s^Av_2\right|\right)}{\left|D_s^Av_2\right|}D_s^Av_2\right)\left(\overline{D_s^Av_1}-\overline{D_s^Av_2}\right)\mathrm d\mu\mathrm d\tau\right]\\
&=\Re\left[\int_0^t\int_\Omega^{}\left(\left|z_1\right|^{p-1}z_1-\left|z_2\right|^{p-1}z_2\right)\overline{\widehat v}\mathrm dx\mathrm d\tau\right].
\end{split}
\end{equation}
From Lemma \ref{Estimate of nonlinear term}, making use of H$\mathrm{\ddot{o}}$lder, Sobolev and Young's inequalities, we obtain
\begin{equation}\label{021}
\begin{split}
&\left|\int_0^t\int_\Omega\left(\left|z_1\right|^{p-1}z_1-\left|z_2\right|^{p-1}z_2\right)\overline{\widehat v}\mathrm dx\mathrm d\tau\right|\\
&\leq\int_0^t\int_\Omega p\left(\left|z_1\right|+\left|z_2\right|\right)^{p-1}\left|z_1-z_2\right|\left|\widehat v\right|\mathrm dx\mathrm d\tau\\
&\leq C\int_0^t\left(\int_\Omega\left(\left|z_1\right|+\left|z_2\right|\right)^{\left(p-1\right)A_1}\mathrm dx\right)^\frac1{A_1}{\left\|z_1-z_2\right\|}_{A_2}\left\|\widehat v\right\|_{A_3}\mathrm d\tau\\
&\leq C\int_0^t\left(\int_\Omega\left(\left|z_1\right|+\left|z_2\right|\right)^{\left(p-1\right)A_1}\mathrm dx\right)^\frac1{A_1}{\left\|z_1-z_2\right\|}_{A_2}\left[\widehat v\right]_{s,2}\mathrm d\tau\\
&\leq C\int_0^t\left\|\left|z_1\right|+\left|z_2\right|\right\|_{\left(p-1\right)A_1}^{\left(p-1\right)}{\left\|z_1-z_2\right\|}_{A_2}\left[\widehat v\right]_{s,2}\mathrm d\tau\\
&\leq CR^{2\left(p-1\right)}\int_0^t\left[z_1-z_2\right]_{s,2}^2\mathrm d\tau+\frac12\int_0^t\left[\widehat v\right]_{s,2}^2\mathrm d\tau,
\end{split}
\end{equation}
where $A_1=\frac{N}{2s}$, ${A_2} = {A_3} =\frac{2N}{N - 2s}$ by $p\leq\frac{N+2s}{N-2s}$.

From (\ref{021(1)}) and (\ref{021}), we have
\begin{equation*}
\left\|\widehat v\right\|_{s,2,0}^2\leq CR^{2\left(p-1\right)}\int_0^t\left[z_1-z_2\right]_{s,2}^2\mathrm d\tau+\int_0^t\left[\widehat v\right]_{s,2}^2\mathrm d\tau,
\end{equation*}
which gives
\begin{equation*}
\begin{split}
&\left\|\Lambda\left(z_1\right)-\Lambda\left(z_2\right)\right\|_{\mathcal H}^2=\left\|\widehat v\right\|_{\mathcal H}^2=\max_{t\in\left[0,T\right]}\left[\widehat v\right]_{s,2}^2\\
&\leq CR^{2\left(p-1\right)}\int_0^t\left[z_1-z_2\right]_{s,2}^2\mathrm d\tau+\int_0^t\left[\widehat v\right]_{s,2}^2\mathrm d\tau\\
&\leq CR^{2\left(p-1\right)}T\underset{t\in\left[0,T\right]}{\max}\left[z_1-z_2\right]_{s,2}^2+T\underset{t\in\left[0,T\right]}{\max}\left[\widehat v\right]_{s,2}^2,
\end{split}
\end{equation*}
i.e.,
\begin{equation*}
\begin{split}
&\left\|\Lambda\left(z_1\right)-\Lambda\left(z_2\right)\right\|_{\mathcal H}^2\leq\frac{CR^{2\left(p-1\right)}T}{1-T}\underset{t\in\left[0,T\right]}{\max}\left[z_1-z_2\right]_{s,2}^2\\
&\;\;\;\;\;\;\;\;\;\;\;\;\;\;\;\;\;\;\;\;\;\;\;\;\;\;\;\;\;\overset\triangle=\delta_T\left\|z_1-z_2\right\|_{\mathcal H}^2,
\end{split}
\end{equation*}
for some $\delta_T<1$ as long as $T$ is sufficiently small.  Thus the map $v = \Lambda \left( {u} \right)$ is contractive.  Using
the Contraction Mapping Principle, there exists a unique local weak solution to (\ref{1.1}) defined on $\left[0,T\right]$. Theorem \ref{theorem0} is proved.

\section{Global existence and uniqueness}\label{Global existence and uniqueness}
In this section, we deal with the global existence and uniqueness of the weak solution to (\ref{1.1}) under the subcritical initial energy, the critical initial energy and the supercritical initial energy, respectively.

\textbf{Proof of Theorem \ref{theorem1}.}

(i) {\bf Global existence.} Similar to the proof process of Theorem \ref{theorem0}, let ${\left\{\omega_j\left(x\right)\right\}}_{j\in\mathbb{N}}$ be the orthogonal basis in $ W_{A,0}^{s,G}\left(\Omega,\mathbb{C}\right)\cap H_0^s\left(\Omega,\mathbb{C}\right)$.  We look for the approximate solutions of the following form
\begin{equation*}
u_m\left(x,t\right)=\sum_{j=1}^mg_{jm}\left(t\right)\omega_j(x),\;\;\;m=1,2,\cdots,
\end{equation*}
where $g_{jm}\left(t\right):\left[0,T\right]\rightarrow\mathbb{R}$ satisfies the system of ODEs
\begin{equation}\label{3.1.1}
\begin{split}
&\Re \left[ \int_\Omega  {u_{mt}} \overline {\omega _s}\mathrm dx +\int_QD_s^{}u_{mt}\left(x,y\right)\overline{D_s^{}\omega_s\left(x,y\right)}\mathrm d\mu\right.\\
&\left.+ \int_Q {\frac{{g\left( {\left| {D_s^A{u_m}\left( {x,y} \right)} \right|} \right)}}{{\left| {D_s^A{u_m}\left( {x,y} \right)} \right|}}} D_s^A{u_m}\left( {x,y} \right)\overline {D_s^A{\omega _s}\left( {x,y} \right)}  \mathrm d\mu \right]= \Re \left[ \int_\Omega  f\left( u_m\right){\overline {\omega _s}}\mathrm dx\right],
\end{split}
\end{equation}
\begin{equation}\label{3.1.2}
u_m\left(x,0\right)=\sum_{j=1}^ma_{jm}\omega_j(x)\rightarrow u_0(x)\;\mathrm{in}\;W_{A,0}^{s,G}\left(\Omega,\mathbb{C}\right)\cap {H_0^s}\left( \Omega,\mathbb{C}\right),
\end{equation}
for $1\leq s\leq m$, in which
\begin{equation*}
a_{jm}=g_{jm}{(0)}\;\mathrm{and}\;f\left(u\right)=\left|u\right|^{p-1}u.
\end{equation*}

Multiplying the equation of (\ref{3.1.1})  by $g_{jm}\left(t\right)$  and summing for $j$, we have
\begin{equation}\label{3.1.3}
\frac12\frac d{dt}\left\|u_m\right\|_{s,2,0}^2+I\left(u_m\right)=0.
\end{equation}

Observe that by Proposition \ref{proposition}, we obtain $u_m\in W$, i.e., $I\left(u_m\right)>0$. Then from (\ref{3.1.3}) we have
\begin{equation*}
\frac12\frac d{dt}\left\|u_m\right\|_{s,2,0}^2<0,
\end{equation*}
which applying Gronwall's inequality gives
\begin{equation}\label{3.1.3a}
  \left\|u_m\right\|_{s,2,0}^2<C,\;\;\;0\leq t<\infty.
\end{equation}
Multiplying both sides of (\ref{3.1.1}) by $g'_{jm}\left(t\right)$ and summing on $j$, integrating with respect to $t$ from $0$ to $t$, we obtain
\begin{equation}\label{3.1.4}
\int_0^t\left\|u_{m\tau}\right\|_{s,2,0}^2\operatorname d\tau+J\left(u_m\right)=J\left(u_m\left(0\right)\right).
\end{equation}

From (\ref{3.1.2}), we have $J\left(u_m\left(0\right)\right)\rightarrow J\left(u_0\right)<d$ and $I\left(u_m\left(0\right)\right)\rightarrow I\left(u_0\right)>0$, which combining (\ref{3.1.4}) yields that
\begin{equation}\label{3.1.5}
\int_0^t\left\|u_{m\tau}\right\|_{s,2,0}^2\mathrm d\tau+J\left(u_m\right)<d.
\end{equation}

Combining (\ref{3.1.5}) and
\begin{equation}\label{3.1.5b}
J\left( u_m \right) \ge \left( {\frac{1}{q^ + } - \frac{1}{{p + 1}}} \right) \int_Q g \left( {\left| {D_s^Au_m\left( {x,y} \right)} \right|} \right)\left| {D_s^Au_m\left( {x,y} \right)} \right|\mathrm d\mu+ \frac{1}{p + 1}I\left( {{u_m}} \right),
\end{equation}
we obtain
\begin{equation*}
\int_0^t {\left\| {u_{m\tau }} \right\|_{s,2,0}^2}\mathrm d\tau  + \left( {\frac{1}{q^ + } - \frac{1}{p + 1}} \right)\int_Q g {\left| {D_s^Au\left( {x,y} \right)} \right|}\left| {D_s^Au\left( {x,y} \right)} \right|\mathrm d\mu< d,
\end{equation*}
where $m$ is sufficiently large, which implies a priori estimate
\begin{equation}\label{3.1.5a}
\begin{split}
&\int_0^t\left\|u_{m\tau}\right\|_{s,2,0}^2\mathrm d\tau<d,\\
&\int_Qg\left(\left|D_s^Au\left(x,y\right)\right|\right)\left|D_s^Au\left(x,y\right)\right|\mathrm d\mu<\frac{p+1-q^+}{q^+\left(p+1\right)}d,\\
&\left\|u_m^p\right\|_\gamma^\gamma=\left\|u_m\right\|_{p+1}^{p+1}\leq\frac{C_\ast^{p+1}}{\left(q^-\right)^{\frac{p+1}{q^-}}}\max\left\{\left(\frac{ p+1- q^+}{ q^+\left( p+1\right)} d\right)^\frac{ p+1}{q^+},\left(\frac{ p+1-q^+}{q^+\left(p+1\right)} d\right)^\frac{p+1}{q^-}\right\},\;\gamma=\frac{p+1}p.
\end{split}
\end{equation}
An estimation process similar to (\ref{013a}), we have
\begin{equation}\label{3.1.5b}
{\left\|\frac{g\left(\left|D_s^Au_m\right|\right)}{\left|D_s^Au_m\right|}D_s^Au_m\right\|}_{\widetilde G}\leq C\max\left\{\left(\frac{\left(p+1-q^+\right)}{q^+\left(p+1\right)}d\right)^\frac1{q^+},\left(\frac{\left(p+1-q^+\right)}{q^+\left(p+1\right)}d\right)^\frac1{q^-}\right\}.
\end{equation}

Therefore, from (\ref{3.1.3a}), (\ref{3.1.5a}) and (\ref{3.1.5b}), there exist a $u$ and a subsequence of $\left\{u_m\right\}$, which is still denoted by
itself, such that
\begin{equation*}
\left\{\begin{split}
&u_m\rightharpoonup u\;\;\mathrm{in}\;L^\infty\left(0,\infty;W_{A,0}^{s,G}\left(\Omega,\mathbb{C}\right)\cap H_0^s\left(\Omega,\mathbb{C}\right)\right)\;\mathrm{weak}\;\mathrm{star},\\
&\frac{g\left(\left|D_s^Au_m\left(x,y\right)\right|\right)}{\left|D_s^Au_m\left(x,y\right)\right|}D_s^Au_m\left(x,y\right)\rightharpoonup\xi\;\mathrm{in}\;L^\infty\left(\left[0,T\right];L_\mu^{G^\ast}\left(\Omega,\mathbb{C}\right)\right)\;\mathrm{weak}\;\mathrm{star},\\
&u_{mt}\rightharpoonup u_t\;\;\mathrm{in}\;L^2\left(0,\infty; H_0^s\left(\Omega,\mathbb{C}\right)\right)\;\mathrm{weak},\\
&\left|u_m\right|^{p-1}u_m\rightharpoonup\left|u\right|^{p-1}u\;\;\mathrm{in}\;L^\infty\left(0,\infty; L^\frac{p+1}p(\Omega,\mathbb{C})\right)\;\mathrm{weak}\;\mathrm{star}.
\end{split}
\right.
\end{equation*}

For $s$ fixed and  we can pass to the limit in (\ref{3.1.1}) to get
\begin{equation*}
\Re\left[\int_\Omega u_t{\overline\omega}_s\mathrm dx+\int_QD_s^{}u_t\left(x,y\right)\overline{D_s^{}\omega_s\left(x,y\right)}\mathrm d\mu+\int_Q\xi\overline{D_s^A\omega_s\left(x,y\right)}\mathrm d\mu\right]=\Re\left[\int_\Omega f\left(u\right){\overline\omega}_s\mathrm dx\right]
\end{equation*}
and
\begin{equation*}
\begin{split}
&\Re\left[\int_\Omega u_t\overline v\mathrm dx+\int_QD_s^{}u_t\left(x,y\right)\overline{D_s^{}v\left(x,y\right)}\mathrm d\mu+\int_Q\xi\overline{D_s^Av\left(x,y\right)}\mathrm d\mu\right]\\
&=\Re\left[\int_\Omega f\left(u\right)\overline vdx\right]\;\mathrm{for}\;\mathrm{all}\;v\in W_{A,0}^{s,G}\left(\Omega,\mathbb{C}\right)\cap H_0^s\left(\Omega,\mathbb{C}\right),\;t\in\left(0,\infty\right).
\end{split}
\end{equation*}
Similar to the proof of Theorem \ref{theorem0}, we have
\begin{equation*}
\Re\left[\int_Q\xi\overline{D_s^Av\left(x,y\right)}\mathrm d\mu\right]=\Re\left[\int_Q\frac{g\left(\left|D_s^Au\left(x,y\right)\right|\right)}{\left|D_s^Au\left(x,y\right)\right|}D_s^Au\left(x,y\right)\overline{D_s^A\varphi\left(x,y\right)}\mathrm d\mu\right].
\end{equation*}
In addition, (\ref{3.1.2}) gives $u(x,0)=u_0(x)$ in $W_{A,0}^{s,G}\left(\Omega,\mathbb{C}\right)\cap H_0^s\left(\Omega,\mathbb{C}\right)$. The global existence is proved.

{\bf Uniqueness.} Assuming that $\eta_1$ and $\eta_2$ are two weak solutions to (\ref{1.1}) with the same initial data. We define  $v=\eta_1-\eta_2$  and obtain
\begin{equation*}
\eta_1,\eta_2,v\in L^\infty\left(0,\infty;W_{A,0}^{s,G}\left(\Omega,\mathbb{C}\right)\cap H_0^s\left(\Omega,\mathbb{C}\right)\right),\;\eta_{1t},\eta_{2t},v_t\in L^2\left(0,\infty;H_0^s\left(\Omega,\mathbb{C}\right)\right),
\end{equation*}
\begin{equation}\label{3.1.7}
v_t+\left(-\Delta\right)^sv_t+\left(-\Delta_g^A\right)^s\eta_1-\left(-\Delta_g^A\right)^s\eta_2=\left|\eta_1\right|^{p-1}\eta_1-\left|\eta_2\right|^{p-1}\eta_2.
\end{equation}

Multiplying (\ref{3.1.7}) by $v$ and integrating over $\left(0,t\right)\times\Omega$, we have
\begin{equation}\label{3.1.8}
\begin{split}
&\Re\left[\int_0^t\int_\Omega{\left(v_\tau\overline v+\left(-\Delta\right)^sv_t\overline v+\left(\left(-\Delta_g^A\right)^s\eta_1-\left(-\Delta_g^A\right)^s\eta_2\right)\overline v\right)}\mathrm dx\mathrm d\tau\right]\\
&=\Re\left[\int_0^t\int_\Omega^{}\left(\vert \eta_1\vert^{p-1}\eta_1-\vert \eta_2\vert^{p-1}\eta_2\right)\overline v\mathrm dx\mathrm d\tau\right].
\end{split}
\end{equation}

By Lemma \ref{Estimate of nonlinear term}, we have
\begin{equation}\label{3.1.9}
\begin{split}
&\left|\int_0^t\int_\Omega\left(\left|\eta_1\right|^{p-1}\eta_1-\left|\eta_2\right|^{p-1}\eta_2\right)\overline v\mathrm dx\mathrm d\tau\right|\\
&\leq\int_0^t\int_\Omega p\left(\left|\eta_1\right|+\left|\eta_2\right|\right)^{p-1}\left|\eta_1-\eta_2\right|\left|v\right|\mathrm dx\mathrm d\tau\\
&\leq C\int_0^t{\left\|\left(\left|\eta_1\right|+\left|\eta_2\right|\right)^{p-1}\right\|}_{B_1}{\left\|\eta_1-\eta_2\right\|}_{B_2}{\left\|v\right\|}_{B_3}\mathrm d\tau\\
&=C\int_0^t\left\|\left|\eta_1\right|+\left|\eta_2\right|\right\|_{{\left(p-1\right)}B_1}^{p-1}{\left\|\eta_1-\eta_2\right\|}_{B_2}{\left\|v\right\|}_{B_3}\mathrm d\tau\\
&\leq C\int_0^t{\left\|\eta_1-\eta_2\right\|}_{s,2,0}{\left\|v\right\|}_{s,2,0}\mathrm d\tau\\
&\leq C\int_0^t\left\|v\right\|_{s,2,0}^2\mathrm d\tau,
\end{split}
\end{equation}
where $B_1=\frac N{2s}$ and $B_2=B_3=\frac{2N}{N-2s}$ by $p\leq\frac{sq^-+N}{N-sq^-}$. Noticing
that $v(x,0)=0$, we get
\begin{equation}\label{3.1.10}
\Re\left[\int_0^t\int_\Omega \left(v_\tau \overline v\mathrm dx+\left(-\Delta\right)^sv_t\overline v\right)\mathrm d\tau\right]=\frac12\left\|v\left(t\right)\right\|_{s,2,0}^2.
\end{equation}
On the other hand, we obtain
\begin{equation}\label{3.1.11}
\begin{split}
&\Re\left[\int_\Omega\left(\left(-\Delta_g^A\right)^s\eta_1\left(x\right)-\left(-\Delta_g^A\right)^s\eta_2\left(x\right)\right)\overline v\left(x\right)\mathrm dx\right]\\
&=\Re\left[\int_Q\left(\frac{g\left(\left|D_s^A\eta_1\left(x,y\right)\right|\right)}{\left|D_s^A\eta_1\left(x,y\right)\right|}D_s^A\eta_1\left(x,y\right)-\frac{g\left(\left|D_s^A\eta_2\left(x,y\right)\right|\right)}{\left|D_s^A\eta_2\left(x,y\right)\right|}D_s^A\eta_2\left(x,y\right)\right)\overline{D_s^Av\left(x,y\right)}\mathrm d\mu\right]\\
&\geq0.
\end{split}
\end{equation}
Then from (\ref{3.1.8})-(\ref{3.1.11}), we find
\begin{equation*}
\left\|v\right\|_{s,2,0}^2\leq C\int_0^t\left\|v\right\|_{s,2,0}^2\mathrm d\tau.
\end{equation*}
It follows from Gronwall's inequality that $\left\|v\right\|_{s,2,0}^2\leq0$, i.e., $\left\|v\right\|_{s,2,0}^2=\left\|\eta_1-\eta_2\right\|_{s,2,0}^2=0$. Thus $\eta_1=\eta_2=0$ a.e. in $\Omega\times\left(0,\infty\right)$. The uniqueness is proved.

(ii) Let $\mu_s=1-\frac1s$ and $u_{s0}=\mu_su_0$ for $s=2,3,\cdots$. We consider the problem (\ref{1.1}) with the condition
\begin{equation}\label{3.2.1}
u\left(x,0\right)=u_{s0}(x),\;s=2,3,\cdots.
\end{equation}

We assert $I(u_{s0})>0$ and $J(u_{s0})<d$. In fact, by the condition $J\left(u_0\right)=d$,  we know $\rho_{s,G}^A\left(u_0\right)\neq0$.  If $\int_\Omega\left|u_0\right|^{p+1}\operatorname dx=0$, combining
\begin{equation*}
  J\left(u_0\right)\leq\left(\frac1{q^-}-\frac1{p+1}\right)\int_Qg\left(\left|D_s^Au_0\left(x,y\right)\right|\right)\left|D_s^Au_0\left(x,y\right)\right|\mathrm d\mu+\frac1{p+1}I\left(u_0\right)
\end{equation*}
and $J\left(u_0\right)=d$, we have $I\left(u_0\right)\geq\left(p+1\right)d$. Then from (\ref{2.1.6}) and (\ref{2.1.7}), we obtain
\begin{equation*}
\begin{split}
&I(u_{s0})=I(\mu_su_0)=\int_Qg\left(\left|\mu_sD_s^Au_0\left(x,y\right)\right|\right)\left|\mu_sD_s^Au_0\left(x,y\right)\right|\mathrm d\mu\\
&\;\;\;\;\;\;\;\;\;\;\;\;\;\;\;\;\;\;\;\;\;\;\;\;\;\;\geq \min\left\{\mu_s^{q^-},\mu_s^{q^+}\right\}\int_Qg\left(\left|D_s^Au_0\left(x,y\right)\right|\right)\left|D_s^Au_0\left(x,y\right)\right|\mathrm d\mu\\
&\;\;\;\;\;\;\;\;\;\;\;\;\;\;\;\;\;\;\;\;\;\;\;\;\;\;=\min\left\{\mu_s^{q^-},\mu_s^{q^+}\right\}I\left(u_0\right)\\
&\;\;\;\;\;\;\;\;\;\;\;\;\;\;\;\;\;\;\;\;\;\;\;\;\;\;>0,
\end{split}
\end{equation*}
and
\begin{equation*}
\begin{split}
&J(u_{s0})=J(\mu_su_0)=\int_QG\left(\left|\mu_sD_s^Au_0\left(x,y\right)\right|\right)\mathrm d\mu\\
&\;\;\;\;\;\;\;\;\;\;\;\;\;\;\;\;\;\;\;\;\;\;\;\;\;\;\leq \max\left\{\mu_s^{q^-},\mu_s^{q^+}\right\}\int_QG\left(\left|D_s^Au_0\left(x,y\right)\right|\right)\mathrm d\mu\\
&\;\;\;\;\;\;\;\;\;\;\;\;\;\;\;\;\;\;\;\;\;\;\;\;\;\;=\max\left\{\mu_s^{q^-},\mu_s^{q^+}\right\}d\\
&\;\;\;\;\;\;\;\;\;\;\;\;\;\;\;\;\;\;\;\;\;\;\;\;\;\;<d.
\end{split}
\end{equation*}
If $\int_\Omega\left|u_0\right|^{p+1}\operatorname dx\neq0$, then from Lemma \ref{Properties of J} and the initial data $I(u_0)\geq0$,  we have $\lambda^\ast\geq1$. Then, $I(u_{s0})=I(\mu_su_0)>0$ and $J(u_{s0})=J(\mu_su_0)<J(u_0)=d$.

Using the similar arguments as in (i), we know that problem (\ref{1.1}) with initial condition (\ref{3.2.1}) has a unique global solution $u_s\in L^\infty\left(0,\infty;W_{A,0}^{s,G}\left(\Omega,\mathbb{C}\right)\cap H_0^s\left(\Omega,\mathbb{C}\right)\right)$  with $u_{st}\in L^2\;(0,\infty;H_0^s(\Omega,\mathbb{C}))$ for each $s=2,3,\cdots$. According to Proposition \ref{proposition}, we can easily know that $u_s\in W$. Similar to the proof in (i), we can easily obtain
\begin{equation*}
  \left\|u_s\right\|_{s,2,0}^2<C,
\end{equation*}
\begin{equation*}
\int_0^t\left\|u_{s\tau}\right\|_{s,2,0}^2\mathrm  d\tau<d,
\end{equation*}
\begin{equation*}\label{3.2.2}
\int_Qg\left(\left|D_s^Au_s\left(x,y\right)\right|\right)\left|D_s^Au_s\left(x,y\right)\right|\mathrm d\mu<\frac{p+1-q^+}{q^+\left(p+1\right)}d,
\end{equation*}
\begin{equation*}\label{3.2.3}
\left\|u_s^p\right\|_\gamma^\gamma=\left\|u_s\right\|_{p+1}^{p+1}\leq\frac{C_\ast^{p+1}}{\left(q^-\right)^{\frac{p+1}{q^-}}}\max\left\{\left(\frac{ p+1- q^+}{ q^+\left( p+1\right)} d\right)^\frac{ p+1}{q^+},\left(\frac{ p+1-q^+}{q^+\left(p+1\right)} d\right)^\frac{p+1}{q^-}\right\},\;\gamma=\frac{p+1}p,
\end{equation*}
and
\begin{equation*}
{\left\|\frac{g\left(\left|D_s^Au_s\right|\right)}{\left|D_s^Au_s\right|}D_s^Au_s\right\|}_{\widetilde G}\leq C\max\left\{\left(\frac{\left(p+1-q^+\right)}{q^+\left(p+1\right)}d\right)^\frac1{q^+},\left(\frac{\left(p+1-q^+\right)}{q^+\left(p+1\right)}d\right)^\frac1{q^-}\right\}.
\end{equation*}
The rest is proved similar to (i).

(iii) Consider the basis $u_m\left(x,t\right)$ as (i), we assert that $u_m\in{\mathcal N}_+$.

If it is false, there exists a $t_0>0$ such that $u_m\in{\mathcal N}_+$ for $t\in\left(0,t_0\right)$ and $u_m\left(t_0\right)\in{\mathcal N}$. Taking $\varphi=u_m$,
we obtain
\begin{equation*}
\frac12\frac d{dt}\left\|u_m\right\|_{s,2,0}^2+\int_Qg\left(\left|D_s^Au_m\left(x,y\right)\right|\right)\left|D_s^Au_m\left(x,y\right)\right|\mathrm d\mu=\int_\Omega\left|u_m\right|^{p+1}\mathrm dx,
\end{equation*}
Thus,
\begin{equation}\label{3.3.1}
\frac12\frac d{dt}\left\|u_m\right\|_{s,2,0}^2=-I\left(u_m\right),
\end{equation}
which means that $\left\|u_m\right\|_{s,2,0}^2$ are bounded for $\Omega\times\left(0,t_0\right)$. Hence, we deduce that $u_m(t_0)\in J^{J\left(u_m\left(0\right)\right)}$ from (\ref{3.1.4}). Therefore, $u_m(t_0)\in\mathcal N^{J\left(u_m\left(0\right)\right)}$.By the definition of $\lambda_{J(u_0)}$, we get
\begin{equation}\label{3.3.2}
  \left\|u_m\left(t_0\right)\right\|_{s,2,0}^2\geq\lambda_{J(u_0)}.
\end{equation}

Based on  fact $I\left(u_m\right)>0$ and combined with (\ref{3.3.1}), we reach
\begin{equation*}
\left\|u_m\left(t_0\right)\right\|_{s,2,0}^2<\left\|u_m\left(x,0\right)\right\|_{s,2,0}^2\leq\lambda_{J(u_0)},
\end{equation*}
which contradicts (\ref{3.3.2}). Therefore, for all $t_0>0$, $u_m\in{\mathcal N}_+$.

Furthermore, we can obtain $u_m\left(t\right)\in J^{J\left(u_m\left(0\right)\right)}\cap{\mathcal N}_+$. On the other hand, it can be inferred from (\ref{3.1.4}) that
\begin{equation*}
\begin{split}
&J\left(u_m\left(0\right)\right)\geq J\left(u_m\left(t\right)\right) \\
&\ge \left( {\frac{1}{q^ + } - \frac{1}{p + 1}} \right)\int_Q g \left( {\left| {D_s^A{u_m}\left( {x,y} \right)} \right|} \right)\left| {D_s^A{u_m}\left( {x,y} \right)} \right|\mathrm d\mu+ \frac{1}{p + 1}I\left( {u_m} \right)
\end{split}
\end{equation*}
and
\begin{equation*}
  J\left(u_m\left(0\right)\right)\geq\int_0^t\left\|u_{m\tau}\right\|_{s,2,0}^2\operatorname d\tau,
\end{equation*}
which means that $\left( {\frac{1}{q^ + } - \frac{1}{p + 1}} \right)\int_Q g \left( {\left| {D_s^A{u_m}\left( {x,y} \right)} \right|} \right)\left| {D_s^A{u_m}\left( {x,y} \right)} \right|\mathrm d\mu$ and $\int_0^t\left\|u_{m\tau}\right\|_{s,2,0}^2\operatorname d\tau$ are bounded. Similar to the proof process in (i), the problem (\ref{1.1}) has a unique global weak solution $u$.

Theorem \ref{theorem1} is proved.

\section{Asymptotic behavior}\label{Asymptotic behavior}

In this section, we shall prove the asymptotic behavior of solutions to problem (\ref{1.1}) under sub-sharp-critical initial energy, sharp-critical initial energy and supercritical initial energy, respectively.

The following lemma gives the key inequality required by the theorem, and is the reason why we need to take the $J\left(u_0\right)<M$ condition instead of $J\left(u_0\right)<d$.

\begin{lemma}\label{Inequality for J<M}
Let $J\left(u_0\right)<M$, $I\left(u\right)>0$ and $u(t)$ be a weak solution of (\ref{1.1}), then
\begin{equation*}
\int_Q g \left( {\left| {D_s^Au\left( {x,y} \right)} \right|} \right)\left| {D_s^Au\left( {x,y} \right)} \right|{\rm{d}}\mu <1.
\end{equation*}
\end{lemma}
\begin{proof}
From  the definitions of $J\left(u\right)$ and $I\left(u\right)$  and $I\left(u\right)>0$, we have
\begin{equation*}
  J(u)> \left( {\frac{1}{q^ + } - \frac{1}{p + 1}} \right) \int_Q g \left( {\left| {D_s^A{u}\left( {x,y} \right)} \right|} \right)\left| {D_s^A{u}\left( {x,y} \right)} \right|\mathrm d\mu ,
\end{equation*}
which from (\ref{Weak solution equation2}) implies
\begin{equation}\label{4.1.3}
\int_Q g \left( {\left| {D_s^A{u}\left( {x,y} \right)} \right|} \right)\left| {D_s^A{u}\left( {x,y} \right)} \right|\mathrm d\mu  <\frac{{q^ + }\left( {p + 1} \right)}{p + 1 - {q^ + }}J\left(u\right)\leq\frac{{q^ + }\left( {p + 1} \right)}{p + 1 - {q^ + }}J\left(u_0\right).
\end{equation}

By (\ref{4.1.3}) and $J\left(u_0\right)<M$,  we get
\begin{equation}\label{4.1.4}
\begin{split}
&\max\left\{\frac{C_\ast^{p+1}}{\left(q^-\right)^\frac{p+1}{q^-}},1\right\}\left(\int_Q g \left( {\left| {D_s^Au\left( {x,y} \right)} \right|} \right)\left| {D_s^Au\left( {x,y} \right)} \right|\mathrm d\mu\right)^{\frac{p + 1 - {q^ + }}{q^ + }}\\
&<\max\left\{\frac{C_\ast^{p+1}}{\left(q^-\right)^\frac{p+1}{q^-}},1\right\}\left(\frac{{q^ + }\left( {p + 1} \right)}{p + 1 - {q^ + }}J\left( {u_0} \right)\right)^\frac{p + 1 - {q^ + }}{q^ + }\\
&< 1,
\end{split}
\end{equation}
which says $\int_Q g \left( {\left| {D_s^Au\left( {x,y} \right)} \right|} \right)\left| {D_s^Au\left( {x,y} \right)} \right|{\rm{d}}\mu <1$.
\end{proof}

\

\noindent\textbf{Proof of Theorem \ref{theorem2}.}

(i) Recalling Proposition \ref{proposition}, it is easy to know $u\in W_\delta$ for $1\leq\delta<\overline\delta$ or $\delta_1<\delta<\delta_2$ with $\delta_1<1<\delta_2$, and particularly $I(u)>0$. Then from (\ref{3.3.1}) and Lemma \ref{Inequality for J<M}, we have
\begin{equation}\label{4.1.1}
\begin{split}
&\frac12\frac d{dt}\left\|u\right\|_{s,2,0}^2=-I\left(u\right)\\
&=-\int_Qg\left(\left|D_s^Au\left(x,y\right)\right|\right)\left|D_s^Au\left(x,y\right)\right|\mathrm d\mu+\int_\Omega\left|u\right|^{p+1}\mathrm dx\\
&\leq-\int_Qg\left(\left|D_s^Au\left(x,y\right)\right|\right)\left|D_s^Au\left(x,y\right)\right|\mathrm d\mu+\frac{C_\ast^{p+1}}{\left(q^-\right)^\frac{p+1}{q^-}}\left(\int_Qg\left(\left|D_s^Au\left(x,y\right)\right|\right)\left|D_s^Au\left(x,y\right)\right|\mathrm d\mu\right)^\frac{p+1}{q^+}\\
&\leq\left(\frac{C_\ast^{p+1}}{\left(q^-\right)^\frac{p+1}{q^-}}\left(\int_Qg\left(\left|D_s^Au\left(x,y\right)\right|\right)\left|D_s^Au\left(x,y\right)\right|\mathrm d\mu\right)^\frac{p+1-q^+}{q^+}-1\right)\\
&\;\;\;\;\int_Qg\left(\left|D_s^Au\left(x,y\right)\right|\right)\left|D_s^Au\left(x,y\right)\right|\mathrm d\mu.
\end{split}
\end{equation}

From (\ref{4.1.4}), we obtain
\begin{equation}\label{4.1.2}
\max\left\{\frac{C_\ast^{p+1}}{\left(q^-\right)^\frac{p+1}{q^-}},1\right\}\left(\int_Qg\left(\left|D_s^Au\left(x,y\right)\right|\right)\left|D_s^Au\left(x,y\right)\right|\mathrm d\mu\right)^\frac{p+1-q^+}{q^+}-1<0.
\end{equation}

Considering (\ref{4.1.3}), (\ref{4.1.1}) and (\ref{4.1.2}), we get
\begin{equation*}
\begin{split}
&\frac12\frac d{dt}\left\|u\right\|_{s,2,0}^2\\
&\leq C\left(\max\left\{\frac{C_\ast^{p+1}}{\left(q^-\right)^\frac{p+1}{q^-}},1\right\}\left(\int_Qg\left(\left|D_s^Au\left(x,y\right)\right|\right)\left|D_s^Au\left(x,y\right)\right|\mathrm d\mu\right)^\frac{p+1-q^+}{q^+}-1\right)\left[u\right]_{s,2}^2\\
&<
C\left(\max\left\{\frac{C_\ast^{p+1}}{\left(q^-\right)^\frac{p+1}{q^-}},1\right\}\left(\frac{{q^ + }\left( {p + 1} \right)}{p + 1 - {q^ + }}J\left( {u_0} \right)\right)^\frac{p + 1 - {q^ + }}{q^ + }-1\right)\left[u\right]_{s,2}^2.
\end{split}
\end{equation*}
Then,  there exists $\delta>0$ such that
\begin{equation*}
\frac d{dt}\left\|u\right\|_{s,2,0}^2<-\delta_1\left\|u\right\|_{s,2,0}^2,
\end{equation*}
where
\begin{equation*}
-\delta_1:=C\left( \max\left\{\frac{C_\ast^{p+1}}{\left(q^-\right)^\frac{p+1}{q^-}},1\right\}\left(\frac{{q^ + }\left( {p + 1} \right)}{p + 1 - {q^ + }}J\left( {u_0} \right)\right)^\frac{p + 1 - {q^ + }}{q^ + }-1\right).
\end{equation*}
Using Gronwall's inequality, we have
\begin{equation*}\label{4.1.7}
\left\|u\right\|_{s,2,0}^2<\left\|u_0\right\|_{s,2,0}^2e^{-\delta_1 t},\;0\leq t<\infty
\end{equation*}
for $-\delta_1>0$.

From  (\ref{2.1.4}), we get
\begin{equation}\label{4.1.8}
\begin{split}
&J\left(u\right)+\left\|u\right\|_{s,2,0}^2\leq\rho_{s,G}^A\left(u\right)+C\int_Qg\left(\left|D_s^Au\left(x,y\right)\right|\right)\left|D_s^Au\left(x,y\right)\right|\mathrm d\mu\\
&\;\;\;\;\;\;\;\;\;\;\;\;\;\;\;\;\;\;\;\;\;\;\;\leq\left(\frac1{q^-}+C\right)\int_Qg\left(\left|D_s^Au\left(x,y\right)\right|\right)\left|D_s^Au\left(x,y\right)\right|\mathrm d\mu.
\end{split}
\end{equation}

By (\ref{Weak solution equation2}), (\ref{4.1.1}) and (\ref{4.1.8}), we have
\begin{equation*}
\begin{split}
&\frac d{dt}\left(J\left(u\right)+\left\|u\right\|_{s,2,0}^2\right)=-\left\|u_t\right\|_{s,2,0}^2+\frac d{dt}\left\|u\right\|_{s,2,0}^2\leq\frac d{dt}\left\|u\right\|_{s,2,0}^2\\
&\leq\left(\frac{C_\ast^{p+1}}{\left(q^-\right)^\frac{p+1}{q^-}}\left(\int_Qg\left(\left|D_s^Au\left(x,y\right)\right|\right)\left|D_s^Au\left(x,y\right)\right|\mathrm d\mu\right)^\frac{p+1-q^+}{q^+}-1\right)\\
&\;\;\;\;\int_Qg\left(\left|D_s^Au\left(x,y\right)\right|\right)\left|D_s^Au\left(x,y\right)\right|\mathrm d\mu\\
&\leq\left(\frac{C_\ast^{p+1}}{\left(q^-\right)^\frac{p+1}{q^-}}\left(\int_Qg\left(\left|D_s^Au\left(x,y\right)\right|\right)\left|D_s^Au\left(x,y\right)\right|\mathrm d\mu\right)^\frac{p+1-q^+}{q^+}-1\right)\\
&\;\;\;\;\frac{q^-}{1+Cq^-}\left(J\left(u\right)+\left\|u\right\|_{s,2,0}^2\right).
\end{split}
\end{equation*}

Applying Gronwall's inequality, we obtain
\begin{equation}\label{4.1.10}
J\left(u\right)+\left\|u\right\|_{s,2,0}^2<\left(J\left(u_0\right)+\left\|u_0\right\|_{s,2,0}^2\right)e^{-\delta_2 t},
\end{equation}
where
\begin{equation*}
-\delta_2:=\frac{q^-}{1+Cq^-}\left(\frac{C_\ast^{p+1}}{\left(q^-\right)^\frac{p+1}{q^-}}\left(\int_Qg\left(\left|D_s^Au\left(x,y\right)\right|\right)\left|D_s^Au\left(x,y\right)\right|\mathrm d\mu\right)^\frac{p+1-q^+}{q^+}-1\right).
\end{equation*}

Furthermore, by (\ref{2.1.4}), (\ref{inequality2}), (\ref{4.1.3}), (\ref{4.1.10}) and Lemma \ref{Inequality for J<M}, we get
\begin{equation*}
\begin{split}
&{\left\|u\right\|}_{p+1}\leq C_\ast\left[u\right]_{s,G}^A\leq C_\ast \max\left\{\left(\rho_{s,G}^A\right)^{q^-},\left(\rho_{s,G}^A\right)^{q^+}\right\}\\
&\leq C_\ast \max\left\{\left(\frac1{q^-}\int_Qg\left(\left|D_s^Au\left(x,y\right)\right|\right)\left|D_s^Au\left(x,y\right)\right|\mathrm d\mu\right)^{q^-},\right.\\
&\;\;\;\;\left.\left(\frac1{q^-}\int_Qg\left(\left|D_s^Au\left(x,y\right)\right|\right)\left|D_s^Au\left(x,y\right)\right|\mathrm d\mu\right)^{q^+}\right\}\\
&\leq C_\ast\left(\frac1{q^-}\int_Qg\left(\left|D_s^Au\left(x,y\right)\right|\right)\left|D_s^Au\left(x,y\right)\right|\mathrm d\mu\right)^{q^-}\\
&\leq C_\ast\left(\frac{q^+\left(p+1\right)}{q^-\left(p+1-q^+\right)}J\left(u\right)\right)^{q^-}\\
&<C_\ast\left(\frac{q^+\left(p+1\right)}{q^-\left(p+1-q^+\right)}\right)^{q^-}\left(J\left(u_0\right)+\left\|u_0\right\|_{s,2,0}^2\right)^{q^-}e^{-\delta_2 q^-t}.
\end{split}
\end{equation*}

(ii) Since the existence of the global solution $u(t)$ for (\ref{1.1}) has been proven, then we claim that $u\in W$ for $t>0$.  If the assertion of $u$ is false, take $t_0>0$ as the first time such that $I(u(t_0))=0$. By the definition of $M$, we get $J(u(t_0))\geq M$. But,
\begin{equation}\label{4.2.1}
0<J(u(t_0))=M-\int_0^{t_0}\left\|u_\tau\right\|_{s,2,0}^2\mathrm d\tau=d_1\leq M
\end{equation}
for any $t_0>0$.  Then, we get
\begin{equation}\label{4.2.2}
J(u(t_0))=M.
\end{equation}

From (\ref{4.2.1}) and (\ref{4.2.2}), we infer that
\begin{equation*}\label{4.2.3}
\int_0^{t_0}\left\|u_\tau\right\|_{s,2,0}^2\mathrm d\tau=0,
\end{equation*}
namely, $u_t\equiv0$ for $0\leq t\leq t_0$,  recalling (\ref{3.3.1}), which contradicts $I(u_0)>0$. Then we infer that $u\in W$ for $0<t<\infty$.

Taking $t_1>0$ as the initial time, we have $J\left(u\right)<M$ and $I\left(u\right)>0$ for $t>t_1$. The proof process similar to (i), we know that there exist constants $\varsigma_1>0$ and $\varsigma_2>0$ such that
\begin{equation*}
\begin{split}
&\left\|u\right\|_{s,2,0}^2<\left\|u\left(t_1\right)\right\|_{s,2,0}^2e^{-\varsigma_1\left(t-t_1\right)},\\
&J\left(u\right)+\left\|u\right\|_{s,2,0}^2<\left(J\left(u\left(t_1\right)\right)+\left\|u\left(t_1\right)\right\|_{s,2,0}^2\right)e^{-\varsigma_2\left(t-t_1\right)},\\
&{\left\|u\right\|}_{p+1}<C_\ast\left(\frac{q^+\left(p+1\right)}{q^-\left(p+1-q^+\right)}\right)^{q^-}\left(J\left(u\left(t_1\right)\right)+\left\|u\left(t_1\right)\right\|_{s,2,0}^2\right)^{q^-}e^{-\varsigma_2 q^-\left(t-t_1\right)},
\end{split}
\end{equation*}
where $t\in\left(t_1,+\infty\right)$.

(iii) We use $\omega\left(u_0\right)=\underset{t\geq0}\cap\overline{\left\{u\left(s\right):s\geq t\right\}}$ to represent the $\omega$-limit of $u_0$. From (\ref{3.1.4}) and (\ref{3.3.1}), we obtain
\begin{equation*}
  \left\|\omega\right\|_{s,2,0}^2<\left\|u_0\right\|_{s,2,0}^2\leq\lambda_{J\left(u_0\right)},\;J\left(\omega\right)\leq J\left(u_0\right),
\end{equation*}
for all $\omega\in\omega\left(u_0\right)$, which implies that $\omega\not\in\mathcal N^{J\left(u_0\right)}$ and $\omega\in J^{J\left(u_0\right)}$, that is, $\omega\not\in\mathcal N$. Then, $\omega\left(u_0\right)\cap\mathcal N=\varnothing$, which means $\omega\left(u_0\right)=\left\{0\right\}$. Therefore, $u\left(t\right)\rightarrow0$ as $t\rightarrow+\infty$.

Theorem \ref{theorem2} is proved.

\section{Blowup}\label{Blowup}

In this section, we prove that the solution blows up in finite time with the  sub-sharp-critical initial energy and sharp-critical initial energy, respectively.

\begin{lemma}\label{3.1}
If $u(t)$ is a weak solution of (\ref{1.1}) with $J\left(u_0\right)<d$ and $u_0\in V$, then
\begin{equation*}\label{3.29}
\int_Q g \left( {\left| {D_s^Au\left( {x,y} \right)} \right|} \right)\left| {D_s^Au\left( {x,y} \right)} \right|{\rm{d}}\mu > \frac{{{q^ + }\left( {p + 1} \right)}}{{p + 1 - {q^ + }}}M.
\end{equation*}
\end{lemma}
\begin{proof}
Assume that $u(t)$ is a weak solution of (\ref{1.1}) with $J(u_0)<d$ and $I(u_0)<0$, $T$ is the maximal existence time. By Proposition \ref{proposition},  we can know that $u(x,t)\in V_\delta$, namely, $I(u)<0$ for $0<t<T$. Coming back to Lemma \ref{Relations between I and U} $\mathrm{(ii)}$, take $\delta=1$, we obtain
\begin{equation*}
\int_Q g \left( {\left| {D_s^Au\left( {x,y} \right)} \right|} \right)\left| {D_s^Au\left( {x,y} \right)} \right|{\rm{d}}\mu  > \min\left\{\left(\frac{\left(q^-\right)^\frac{p+1}{q^-}}{C_\ast^{p+1}}\right)^\frac{q^-}{p+1-q^-},\left(\frac{\left(q^-\right)^\frac{p+1}{q^-}}{C_\ast^{p+1}}\right)^\frac{q^+}{p+1-q^+}\right\}.
\end{equation*}
From Lemma \ref{Depth d}, we have
\begin{equation*}\label{3.30}
M < \frac{{p + 1 - {q^ + }}}{{{q^ + }\left( {p + 1} \right)}}\int_Q g \left( {\left| {D_s^Au\left( {x,y} \right)} \right|} \right)\left| {D_s^Au\left( {x,y} \right)} \right|{\rm{d}}\mu .
\end{equation*}
\end{proof}
\textbf{Proof of Theorem \ref{theorem3}.}

(i) We shall prove $u(t)$ blowing up in finite time with $u_0\in V$.

Arguing by contradiction, assume that the solution is global in time. Taking into account $u_0\in V$ and Proposition \ref{proposition}, we obtain $u\in V_\delta$ for $t\in [0, +\infty)$. We define a functional as follows
\begin{equation*}\label{3.31}
F\left(t\right):=\int_0^t\left\|u\right\|_{s,2,0}^2\mathrm d\tau+\left(T_0-t\right)\left\|u_0\right\|_{s,2,0}^2,\;t\in\left[0,T_0\right],
\end{equation*}
where $0<T_0<+\infty$. Obviously, for any $t\in [0,T_0]$ we get $F(t)>0$. Applying the continuity of $F(t)$ with respect to $t$, we infer that there exists a constant $\theta>0$, for $t\in [0,T_0]$ such that $F(t)\geq\theta$. Then
\begin{equation}\label{3.32}
F'\left(t\right)=\left\|u\right\|_{s,2,0}^2-\left\|u_0\right\|_{s,2,0}^2=2\Re\left[\int_0^t\int_\Omega u\overline{u_\tau}\mathrm dx\mathrm d\tau+\int_0^t\int_QD_su\overline{D_su_\tau}\mathrm dx\mathrm d\tau\right],
\end{equation}
and (\ref{3.3.1}) gives
\begin{equation}\label{3.33}
F''\left(t\right)=2\Re\left[\int_\Omega u\overline{u_\tau}\mathrm dx+\int_QD_su\overline{D_su_\tau}\mathrm d\mu\right]=-2I\left(u\right).
\end{equation}
By (\ref{3.32}) and Cauchy-Schwarz inequality,  we reach
\begin{equation}\label{3.34}
\begin{split}
&\left(F'\left(t\right)\right)^2=4\left\{\Re\left[\int_0^t\int_\Omega u\overline{u_\tau}\operatorname dx\mathrm d\tau+\int_0^t\int_QD_su\overline{D_su_\tau}\operatorname dx\mathrm d\tau\right]\right\}^2\\
&\;\;\;\;\;\;\;\;\;\;\;\;\;=4\left\{\left\{\Re\left[\int_0^t\int_\Omega u\overline{u_\tau}\mathrm dx\mathrm d\tau\right]\right\}^2+\left\{\Re\left[\int_0^t\int_QD_su\overline{D_su_\tau}\mathrm dx\mathrm d\tau\right]\right\}^2\right.\\
&\;\;\;\;\;\;\;\;\;\;\;\;\;\;\;\;\;\left.+2\left\{\Re\left[\int_0^t\int_\Omega u\overline{u_\tau}\mathrm dx\mathrm d\tau\right]\right\}\left\{\Re\left[\int_0^t\int_QD_su\overline{D_su_\tau}\mathrm dx\mathrm d\tau\right]\right\}\right\}\\
&\;\;\;\;\;\;\;\;\;\;\;\;\;\;\leq4\left[\int_0^t\left\|u\right\|_2^2\mathrm d\tau\int_0^t\left\|u_\tau\right\|_2^2\mathrm d\tau+\int_0^t\left[u\right]_{s,2}^2\mathrm d\tau\int_0^t\left[u_\tau\right]_{s,2}^2\mathrm d\tau\right.\\
&\;\;\;\;\;\;\;\;\;\;\;\;\;\;\;\;\;\;\left.+2\left(\int_0^t\left\|u\right\|_2^2\mathrm d\tau\right)^\frac12\left(\int_0^t\left\|u_\tau\right\|_2^2\mathrm d\tau\right)^\frac12\left(\int_0^t\left[u\right]_{s,2}^2\right)^\frac12\left(\int_0^t\left[u_\tau\right]_{s,2}^2\mathrm d\tau\right)^\frac12\right]\\
&\;\;\;\;\;\;\;\;\;\;\;\;\;\;\leq4\left[\int_0^t\left\|u\right\|_2^2\mathrm d\tau\int_0^t\left\|u_\tau\right\|_2^2\mathrm d\tau+\int_0^t\left[u\right]_{s,2}^2\mathrm d\tau\int_0^t\left[u_\tau\right]_{s,2}^2\mathrm d\tau\right.\\
&\;\;\;\;\;\;\;\;\;\;\;\;\;\;\;\;\;\;\left.+\int_0^t\left\|u\right\|_2^2\mathrm d\tau\int_0^t\left[u_\tau\right]_{s,2}^2\mathrm d\tau+\int_0^t\left[u\right]_{s,2}^2\mathrm d\tau\int_0^t\left\|u_\tau\right\|_2^2\mathrm d\tau\right]\\
&\;\;\;\;\;\;\;\;\;\;\;\;\;\;\leq4\left(\int_0^t\left\|u\right\|_2^2\mathrm d\tau+\int_0^t\left[u\right]_{s,2}^2\mathrm d\tau\right)\left(\int_0^t\left\|u_\tau\right\|_2^2\mathrm d\tau+\int_0^t\left[u_\tau\right]_{s,2}^2\mathrm d\tau\right)\\
&\;\;\;\;\;\;\;\;\;\;\;\;\;\;\leq4\int_0^t\left\|u\right\|_{s,2,0}^2\mathrm d\tau\int_0^t\left\|u_\tau\right\|_{s,2,0}^2\mathrm d\tau\leq4F\left(t\right)\int_0^t\left\|u_\tau\right\|_{s,2,0}^2\mathrm d\tau.
\end{split}
\end{equation}
Considering (\ref{3.33}) and (\ref{3.34}),  it is easy to see that
\begin{equation*}\label{3.35}
\begin{split}
&F''\left(t\right)F\left(t\right)-\frac{p+3}4\left(F'\left(t\right)\right)^2\geq F\left(t\right)\left(F''\left(t\right)-\left(p+3\right)\int_0^t\left\|u_\tau\right\|_{s,2,0}^2\mathrm d\tau\right)\\
&\;\;\;\;\;\;\;\;\;\;\;\;\;\;\;\;\;\;\;\;\;\;\;\;\;\;\;\;\;\;\;\;\;\;\;\;\;\;\;\;\;\;\;=F\left(t\right)\left(-2I\left(u\right)-\left(p+3\right)\int_0^t\left\|u_\tau\right\|_{s,2,0}^2\mathrm d\tau\right).
\end{split}
\end{equation*}
Let
\begin{equation*}\label{3.36}
\xi\left(t\right)=-2I\left(u\right)-\left(p+3\right)\int_0^t\left\|u_\tau\right\|_{s,2,0}^2\mathrm d\tau.
\end{equation*}
Using (\ref{2.1.4}),  it is easy to see that
\begin{equation}\label{3.37}
\begin{split}
& J\left(u\right)\geq\frac1{q^+}\int_Qg\left(\left|D_s^Au\left(x,y\right)\right|\right)\left|D_s^Au\left(x,y\right)\right|\mathrm d\mu-\frac1{p+1}\int_\Omega\left|u\right|^{p+1}\mathrm dx\\
&\;\;\;\;\;\;\;\geq\left(\frac1{q^+}-\frac1{p+1}\right)\int_Qg\left(\left|D_s^Au\left(x,y\right)\right|\right)\left|D_s^Au\left(x,y\right)\right|\mathrm d\mu+\frac1{p+1}I\left(u\right).
\end{split}
\end{equation}
Hence, from (\ref{Weak solution equation2}) and (\ref{3.37}), we can have
\begin{equation*}\label{3.38}
\begin{split}
&\xi\left(t\right)\geq2\left(\frac{p+1}{q^+}-1\right)\int_Qg\left(\left|D_s^Au\left(x,y\right)\right|\right)\left|D_s^Au\left(x,y\right)\right|\mathrm d\mu\\
&\;\;\;\;\;\;\;\;\;\;\;-2\left(p+1\right)J\left(u_0\right)+\left(p-1\right)\int_0^t\left\|u_\tau\right\|_{s,2,0}^2\mathrm d\tau.
\end{split}
\end{equation*}

Now, we shall discuss the following two cases.

Case 1: $0<J(u_0)<M$. By Lemma \ref{3.1}, we find
\begin{equation}\label{3.39}
\xi\left(t\right)\geq2\left(p+1\right)M-2\left(p+1\right)J\left(u_0\right)+\left(p-1\right)\int_0^t\left\|u_\tau\right\|_{s,2,0}^2\mathrm d\tau=\rho>0.
\end{equation}
Then
\begin{equation*}\label{3.40}
F''\left(t\right)F\left(t\right)-\frac{p+3}4\left(F'\left(t\right)\right)^2\geq\theta\rho>0,\;t\in\left[0,T_0\right],
\end{equation*}
which yields
\begin{equation*}\label{3.41}
\left(F^{-\vartheta}\left(t\right)\right)''=-\frac\vartheta{F^{\vartheta+2}\left(t\right)}\left(F''\left(t\right)F\left(t\right)-\left(\vartheta+1\right)\left(F'\left(t\right)\right)^2\right)<0,\;\vartheta=\frac{p-1}4.
\end{equation*}
Similar to the proof of \cite[Theorem 4.3]{Payne}, there exists a $T>0$ such that
\begin{equation*}\label{3.42}
\lim_{t\rightarrow T}F^{-\vartheta}\left(t\right)=0,
\end{equation*}
and
\begin{equation*}\label{3.43}
\lim_{t\rightarrow T}F\left(t\right)=+\infty,
\end{equation*}
which is a contradiction with $T=+\infty$.

Case 2: $J(u_0)\leq0$. It is easy to get (\ref{3.39}) directly. The rest is similar to the proof of Case 1.

(ii) By the continuity of $J(u)$ and $I(u)$ with respect to $t$, there exists a sufficiently small $t_0>0$ such that $J(u(t_0))>0$ and $I(u(t_0))<0$ for $J(u_0)=M>0$ and $I(u_0)<0$. From (\ref{3.33}),  for $0<t\leq t_0$, we can  see that $u_t\neq0$ for $0<t\leq t_0$.  Hence, we
obtain
\begin{equation*}\label{4.9}
J(u(t_0))=M-\int_0^{t_0}\left\|u_\tau\right\|_{s,2,0}^2\mathrm d\tau=d_0<M.
\end{equation*}
Taking $t=t_0$ as the initial time and recalling Proposition \ref{proposition}, we get $u(x,t)\in V$ for $t>t_0$. The rest is proved similar to (i).

The proof of Theorem \ref{theorem3} is complete.

\section{Convergence relations}\label{some convergence relations}
In this section, we prove the convergence relationship between the global solution of (\ref{1.1}) and the corresponding ground state solution.

The following lemma will play an important role in the proof of Theorem \ref{theorem4}.
\begin{lemma}[\cite{Pablo}]\label{Inequality for N-functions}
Let $G$ is an $N$-function that satisfies (\ref{2.1.4}), (\ref{2.1.5}), and $\left(H_1\right)$, then there exists a constant $C>0$, such that for all $a,b\in R^N$, we have
\begin{equation*}
  \left\langle\frac{g\left(\left|a\right|\right)}{\left|a\right|}a-\frac{g\left(\left|b\right|\right)}{\left|b\right|}b,a-b\right\rangle\geq CG\left(\left|a-b\right|\right).
\end{equation*}
\end{lemma}
\textbf{Proof of Theorem \ref{theorem4}.}
Let $u=u(t)$ be a global solution of (\ref{1.1}). From (\ref{Weak solution equation2}) there holds $J\left(u\right)\leq J\left(u_0\right)$ for $t\in\lbrack0,+\infty)$.

(i) We assert that if $t\in\lbrack0,+\infty)$, then $J\left(u\right)\geq0$.

Assuming that the assertion is not tenable, then there exists a $t_0>0$ such that $J\left(u\left(t_0\right)\right)<0$. In fact, according to the definitions of $I$ and $J$, we have
\begin{equation*}
\begin{split}
&J\left(u\right)\geq\frac1{q^+}\int_Qg\left(\left|D_s^Au\left(x,y\right)\right|\right)\left|D_s^Au\left(x,y\right)\right|\mathrm d\mu-\frac1{p+1}\int_\Omega\left|u\right|^{p+1}\mathrm dx\\
&\;\;\;\;\;\;\;=\frac1{q^+}I\left(u\right)+\left(\frac1{q^+}-\frac1{p+1}\right)\int_\Omega\left|u\right|^{p+1}\mathrm dx,
\end{split}
\end{equation*}
which shows that $I\left(u\left(t_0\right)\right)<0$. Thus $u$ will blow up in finite time. This contradicts with the assumption that $u$ is global solution. Then
\begin{equation*}
0\leq J(u)\leq J(u_0),\;t\in\lbrack0,+\infty).
\end{equation*}

From (\ref{Weak solution equation2}), we deduce
\begin{equation*}
\frac d{dt}J\left(u\right)=-\left\|u_t\right\|_{s,2,0}^2,
\end{equation*}
which  implies $J\left(u\right)$ is decreasing with respect to $t$, that is,
\begin{equation*}
\underset{t\rightarrow+\infty}{\mathrm{lim}}J\left(u\left(t\right)\right)=C,
\end{equation*}
where $C\in\lbrack0,\;J\left(u_0\right)\rbrack$. Letting $t\rightarrow+\infty$ in (\ref{Weak solution equation2}), we obtain
\begin{equation}\label{8.12}
\int_0^{+\infty}\left(\left\|u_t\right\|_2^2+\left[u_t\right]_{s,2}^2\right)\mathrm dt=J\left(u_0\right)-C\leq J\left(u_0\right).
\end{equation}

From (\ref{8.12}), there exists an increasing sequence $\left\{t_k\right\}_{k=1}^\infty$ with $t_k\rightarrow+\infty$ as $k\rightarrow\infty$ such that
\begin{equation}\label{8.13}
\underset{k\rightarrow\infty}{\mathrm{lim}}\left\|u_t\left(t_k\right)\right\|_2^2+\left[u_t\left(t_k\right)\right]_{s,2}^2=0.
\end{equation}

On the other hand, from (\ref{Weak solution equation}), for any $\varphi\in W_{A,0}^{s,G}\left(\Omega,\mathbb{C}\right)$, there holds
\begin{equation*}
\begin{split}
&\left\langle J'\left(u\right),\varphi\right\rangle=\frac d{d\tau}J\left(u+\tau\varphi\right)\\
&=\Re\left[\int_Q\frac{g\left(\left|D_s^Au\left(x,y\right)\right|\right)}{\left|D_s^Au\left(x,y\right)\right|}D_s^Au\left(x,y\right)\overline{D_s^A\varphi\left(x,y\right)}\mathrm d\mu-\int_\Omega\left|u\right|^{p-1}u\overline\varphi\operatorname dx\right]\\
&=\Re\left[-\int_\Omega u_{\mathrm t}\overline\varphi\mathrm dx-\int_QD_su_t\overline{D_s\varphi}\mathrm d\mu\right].\\
\end{split}
\end{equation*}

Applying the Sobolev imbedding inequality and the $\mathrm H\ddot{\mathrm o}\mathrm{lder}$ inequality, taking into account (\ref{8.13}), we obtain that as $k\rightarrow\infty$,
\begin{equation}\label{8.14}
\begin{split}
&{\left\|J'\left(u\left(t_k\right)\right)\right\|}_{-s,G^\ast,0}^A=\sup_{\left\|\varphi\right\|_{s,G,0}^A\leq1}\left|\left\langle J'\left(u\left(t_k\right)\right),\varphi\right\rangle\right|\\
&\;\;\;\;\;\;\;\;\;\;\;\;\;\;\;\;\;\;\;\;\;\;\;\;\;\;\;\leq\sup_{\left\|\varphi\right\|_{s,G,0}^A\leq1}\left|\left\langle-u_t\left(t_k\right)-\left(-\Delta\right)^su_t\left(t_k\right),\varphi\right\rangle\right|\\
&\;\;\;\;\;\;\;\;\;\;\;\;\;\;\;\;\;\;\;\;\;\;\;\;\;\;\;\leq\sup_{\left\|\varphi\right\|_{s,G,0}^A\leq1}\left[{\left\|u_t\left(t_k\right)\right\|}_2{\left\|\varphi\right\|}_2+{\left[u_t\left(t_k\right)\right]}_{s,2}{\left[\varphi\right]}_{s,2}\right]\\
&\;\;\;\;\;\;\;\;\;\;\;\;\;\;\;\;\;\;\;\;\;\;\;\;\;\;\;\leq C\left[{\left\|u_t\left(t_k\right)\right\|}_2+{\left[u_t\left(t_k\right)\right]}_{s,2}\right]\rightarrow0.
\end{split}
\end{equation}
Combining (\ref{2.22}) and (\ref{8.14}), we can get
\begin{equation*}
\begin{split}
&\frac1{p+1}\left|I\left(u\left(t_k\right)\right)\right|=\frac1{p+1}\left|\left\langle J'\left(u\left(t_k\right)\right),u\left(t_k\right)\right\rangle\right|\\
&\;\;\;\;\;\;\;\;\;\;\;\;\;\;\;\;\;\;\;\;\;\;\;\;\leq C{\left\|J'\left(u\left(t_k\right)\right)\right\|}_{-s,G^\ast,0}^A{\left\|u\left(t_k\right)\right\|}_{s,G,0}^A\leq C{\left\|u\left(t_k\right)\right\|}_{s,G,0}^A.
\end{split}
\end{equation*}
Applying the above inequality and considering (\ref{Weak solution equation2}),  we find that
\begin{equation*}
\begin{split}
&J\left(u_0\right)+C\left\|u\left(t_k\right)\right\|_{s,G,0}^A\geq J\left(u\left(t_k\right)\right)-\frac{I\left(u\left(t_k\right)\right)}{p+1}\\
&\;\;\;\;\;\;\;\;\;\;\;\;\;\;\;\;\;\;\;\;\;\;\;\;\;\;\;\;\;\;\;\;\;\;\;\geq\frac{p+1-q^+}{q^+\left(p+1\right)}\int_Qg\left(\left|D_s^Au\left(t_k\right)\right|\right)\left|D_s^Au\left(t_k\right)\right|\mathrm d\mu,
\end{split}
\end{equation*}
which implies
\begin{equation}\label{8.15}
\frac{p+1-q^+}{q^+\left(p+1\right)}\int_Qg\left(\left|D_s^Au\left(t_k\right)\right|\right)\left|D_s^A\left(t_k\right)\right|\mathrm d\mu-C\left\|u\left(t_k\right)\right\|_{s,G,0}^A-J\left(u_0\right)\leq0.
\end{equation}

(ii) We assert that
\begin{equation}\label{8.21}
\left\|u\left(t_k\right)\right\|_{s,G,0}^A\leq C,\;k=1,2,\dots.
\end{equation}

Assuming that the assertion is not tenable,  then there is a sufficiently large $\int_Qg\left(\left|D_s^Au\left(t_k\right)\right|\right)\left|D_s^A\left(t_k\right)\right|\mathrm d\mu$ such that
\begin{equation*}
\frac{p+1-q^+}{q^+\left(p+1\right)}\int_Qg\left(\left|D_s^Au\left(t_k\right)\right|\right)\left|D_s^A\left(t_k\right)\right|\mathrm d\mu-C\left\|u\left(t_k\right)\right\|_{s,G,0}^A-J\left(u_0\right)>0,
\end{equation*}
which contradicts with (\ref{8.15}). Then $\int_Qg\left(\left|D_s^Au\left(t_k\right)\right|\right)\left|D_s^A\left(t_k\right)\right|\mathrm d\mu$ is bounded. By (\ref{2.1.4}), $\rho_{s,G}^A\left(u\right)$ is bounded. Hence, from Lemma \ref{inequality1}, (\ref{8.21}) has been proven.

Consequently, there exist $u^\ast\in W_{A,0}^{s,G}\left(\Omega\right)$ and an increasing subsequence of $\left\{t_k\right\}_{k=1}^\infty$, still denoted by $\left\{t_k\right\}_{k=1}^\infty$, for $k\rightarrow\infty$ such that
\begin{equation*}
\left\{\begin{split}
&u\left(t_k\right)\rightharpoonup u^\ast\;\mathrm{weakly}\;\mathrm{in}\;W_{A,0}^{s,G}\left(\Omega\right),\\
&u\left(t_k\right)\rightarrow u^\ast\;\mathrm{strongly}\;\mathrm{in}\;L^{p+1}\left(\Omega\right).
\end{split}\right.
\end{equation*}

(iii) We prove that $u_k\rightarrow u^\ast\;\mathrm{strongly}\;\mathrm{in}\;W_{A,0}^{s,G}\left(\Omega\right)$.

It is easy to know that
\begin{equation}\label{8.31}
\begin{split}
&\left\langle J'\left(u\left(t_k\right)\right),u\left(t_k\right)-u^\ast\right\rangle=\Re\left[\int_Q\frac{g\left(\left|D_s^Au\left(t_k\right)\right|\right)}{\left|D_s^Au\left(t_k\right)\right|}D_s^Au\left(t_k\right)\overline{D_s^A\left(u\left(t_k\right)-u^\ast\right)}\mathrm d\mu\right.\\
&\left.\;\;\;\;\;\;\;\;\;\;\;\;\;\;\;\;\;\;\;\;\;\;\;\;\;\;\;\;\;\;\;\;\;\;\;\;\;\;\;\;\;-\int_\Omega\left|u\left(t_k\right)\right|^{p-1}u\left(t_k\right)\overline{\left(u\left(t_k\right)-u^\ast\right)}\mathrm dx\right],\\
&\left\langle J'\left(u^\ast\right),u\left(t_k\right)-u^\ast\right\rangle=\Re\left[\int_Q\frac{g\left(\left|D_s^Au^\ast\right|\right)}{\left|D_s^Au^\ast\right|}D_s^Au^\ast\overline{D_s^A\left(u\left(t_k\right)-u^\ast\right)}\mathrm d\mu\right.\\
&\left.\;\;\;\;\;\;\;\;\;\;\;\;\;\;\;\;\;\;\;\;\;\;\;\;\;\;\;\;\;\;\;\;\;\;\;\;\;\;-\int_\Omega\left|u^\ast\right|^{p-1}u^\ast\overline{\left(u\left(t_k\right)-u^\ast\right)}\mathrm dx\right].
\end{split}
\end{equation}

Considering (\ref{8.14}) and (\ref{8.21}), we deduce that as $k\rightarrow\infty$,
\begin{equation}\label{8.32}
\begin{split}
&\left|\left\langle J'\left(u\left(t_k\right)\right),u\left(t_k\right)-u^\ast\right\rangle\right|\leq{\left\|J'\left(u\left(t_k\right)\right)\right\|}_{-s,G^\ast,0}^A\left({\left\|u\left(t_k\right)\right\|}_{s,G,0}^A+{\left\|u^\ast\right\|}_{s,G,0}^A\right)\\
&\;\;\;\;\;\;\;\;\;\;\;\;\;\;\;\;\;\;\;\;\;\;\;\;\;\;\;\;\;\;\;\;\;\;\;\;\;\;\leq\left(C+{\left\|u^\ast\right\|}_{s,G,0}^A\right){\left\|J'\left(u\left(t_k\right)\right)\right\|}_{-s,G^\ast,0}^A\rightarrow0.
\end{split}
\end{equation}
Due to $u_k\rightharpoonup u^\ast\;\mathrm{weakly}\;\mathrm{in}\;W_{A,0}^{s,G}\left(\Omega\right)$,  we get that as $k\rightarrow\infty$,
\begin{equation}\label{8.33}
\left\langle J'\left(u^\ast\right),u\left(t_k\right)-u^\ast\right\rangle\rightarrow 0.
\end{equation}
By (\ref{8.31})-(\ref{8.33}), we conclude that as $k\rightarrow\infty$,
\begin{equation}\label{8.34}
\begin{split}
&\left\langle J'\left(u\left(t_k\right)\right)-J'\left(u^\ast\right),u\left(t_k\right)-u^\ast\right\rangle\\
&=\Re\left[\int_Q\left(\frac{g\left(\left|D_s^Au\left(t_k\right)\right|\right)}{\left|D_s^Au\left(t_k\right)\right|}D_s^Au\left(t_k\right)-\frac{g\left(\left|D_s^Au^\ast\right|\right)}{\left|D_s^Au^\ast\right|}D_s^Au^\ast\right)\overline{D_s^A\left(u\left(t_k\right)-u^\ast\right)}\mathrm d\mu\right.\\
&\;\;\;\;\left.-\int_\Omega\left(\left|u\left(t_k\right)\right|^{p-1}u\left(t_k\right)-\left|u^\ast\right|^{p-1}u^\ast\right)\overline{\left(u\left(t_k\right)-u^\ast\right)}\operatorname dx\right]\rightarrow0.
\end{split}
\end{equation}
From the Sobolev imbedding inequality, the $\mathrm H\ddot{\mathrm o}\mathrm{lder}$ inequality and (\ref{8.21}), we have that as $k\rightarrow\infty$,
\begin{equation}\label{8.35}
\begin{split}
&\left|\int_\Omega\left(\left|u\left(t_k\right)\right|^{p-1}u\left(t_k\right)-\left|u^\ast\right|^{p-1}u^\ast\right)\overline{\left(u\left(t_k\right)-u^\ast\right)}\mathrm dx\right|\\
&\leq\left(\left\|u\left(t_k\right)\right\|_{p+1}^p+\left\|u^\ast\right\|_{p+1}^p\right){\left\|u\left(t_k\right)-u^\ast\right\|}_{p+1}\\
&\leq\left(C\left(\left\|u\left(t_k\right)\right\|_{s,G,0}^A\right)^p+\left\|u^\ast\right\|_{p+1}^p\right){\left\|u\left(t_k\right)-u^\ast\right\|}_{p+1}\\
&\leq\left(C+\left\|u^\ast\right\|_{p+1}^p\right){\left\|u\left(t_k\right)-u^\ast\right\|}_{p+1}\rightarrow0.
\end{split}
\end{equation}
By (\ref{8.34}) and (\ref{8.35}), we deduce that as $k\rightarrow\infty$,
\begin{equation*}
\Re\left[\int_Q\left(\frac{g\left(\left|D_s^Au\left(t_k\right)\right|\right)}{\left|D_s^Au\left(t_k\right)\right|}D_s^Au\left(t_k\right)-\frac{g\left(\left|D_s^Au^\ast\right|\right)}{\left|D_s^Au^\ast\right|}D_s^Au^\ast\right)\overline{D_s^A\left(u\left(t_k\right)-u^\ast\right)}\mathrm d\mu\right]\rightarrow0
\end{equation*}

From Lemma \ref{Inequality for N-functions}, we get
\begin{equation*}
\begin{split}
&\Re\left[\int_Q\left(\frac{g\left(\left|D_s^Au\left(t_k\right)\right|\right)}{\left|D_s^Au\left(t_k\right)\right|}D_s^Au\left(t_k\right)-\int_Q\frac{g\left(\left|D_s^Au^\ast\right|\right)}{\left|D_s^Au^\ast\right|}D_s^Au^\ast\right)\overline{D_s^A\left(u\left(t_k\right)-u^\ast\right)}\mathrm d\mu\right]\\
&\geq C\rho_G\left(D_s^A\left(u\left(t_k\right)-u^\ast\right)\right).
\end{split}
\end{equation*}

Hence, $\underset{k\rightarrow\infty}{\mathrm{lim}}\left[u\left(t_k\right)\right]_{s,G}^A=\left[u^\ast\right]_{s,G}^A$ and then by \cite[Proposition 3.32]{Brezis}, we have $u\left(t_k\right)\rightarrow u^\ast$ strongly in $W_{A,0}^{s,G}\left(\Omega\right).$

Therefore, together with (\ref{8.14}), $J'\left(u^\ast\right)=\underset{k\rightarrow\infty}{\mathrm{lim}}J'\left(u\left(t_k\right)\right)=0$, i.e., $u^\ast\in\Phi$.

The proof of Theorem \ref{theorem4} is complete.

\section{Conclusions}\label{Conclusions}
 In this paper, we study the existence and uniqueness of local solutions and global solutions. Also, we investigate the asymptotic behavior and blowup phenomenon of global solutions. Furthermore, we provide the convergence relationship between the ground state solutions and the global solutions.

Compared to previous researches on fractional $p$-Laplacian parabolic equations and pseudo-parabolic equations (e.g. \cite{Cheng1,Liao}), we replace the \emph{fractional $p$-Laplace operator} $\left(-\Delta_p\right)^s$ with the magnetic fraction $g$-Laplacian operator $\left(-\Delta_g^A\right)^s$ for the first time in pseudo-parabolic equations, and give a detailed introduction to the magnetic space.

The imbedding theorems play an important role in this paper. It is worth noting that $\left\|u\right\|_{s,G,0}^A$ is equivalent to $\left[u\right]_{s,G}^{A}$. The spatial region considered in this paper is $\Omega$, while the spatial region considered in \cite{Pablo} is $\mathbb{R}^N$, and the equivalence of $\left\|u\right\|_{s,G,0}^A$ and $\left[u\right]_{s,G}^{A}$ is not given in \cite{Pablo}. In this paper, we follow the steps of \cite{Bonder2}, which proved the Poincar$\acute{\mathrm{e}}$ inequality in $\Omega$, to conclude the equivalence of $\left\|u\right\|_{s,G,0}^A$ and $\left[u\right]_{s,G}^{A}$. In addition, due to the complexity of the definition of magnetic spaces,  various estimates are difficult. This paper overcomes these difficulties well.

In order to make this paper widely applicable, we define $u\left(x\right)$ in the complex field, provide the definition of weak solutions in the complex field and overcome various proof difficulties brought by the complex field.

Further, the pseudo-parabolic equations in \cite{Cheng,Cheng1,Cheng2,Wang} all contain the $\bbint_\Omega\left|u\right|^{q-1}$ $u\mathrm{d}x$ term. However, the value of $D_s^A1$ in this paper depends on $A\left(\frac{x+y}2\right)$ and is not necessarily equal to $0$, the Conservation Law in \cite[Lemma 2]{Wang} may not hold. Therefore, if the $\bbint_\Omega\left|u\right|^{q-1}u\mathrm{d}x$ term is added to the equation (\ref{1.1}), how to properly handle the computational difficulties caused by $\bbint_\Omega\left|u\right|^{q-1}u\mathrm{d}x$ is a major problem, and will be a direction for future works.



\section*{Statements and Declarations}

{\bf Availability of data and material:} Not applicable.

{\bf Competing interests:} The authors declare that they have no competing interests.

{\bf Funding:} This research was supported by the National Natural Science Foundation of China (No. 12071491).

{\bf Author's contributions:} The authors contributed equally to the writing of this paper. All authors read and approved the final manuscript.

{\bf Acknowledgments:} The first author was supported by the China Scholarship Council (No. 202306380199) and would like to thank Professor Peter C. Gibson and York University for her one-year visiting study at York University.




\bibliographystyle{bmc-mathphys}

\end{document}